\documentclass[a4paper,12pt]{article}
\usepackage{ucs}
\usepackage{amsfonts, amssymb, amsmath, amsthm, amscd}
\usepackage{graphicx}
\usepackage{cite}
\ifx\undefined \pdfgentounicode \else
\input{glyphtounicode} \pdfgentounicode=1
\fi

\author{A.A. Vasil'eva\footnote{Faculty of mechanics and mathematics, Lomonosov Moscow State University; Moscow Center for Fundamental and Applied Mathematics.}}
\title{Kolmogorov widths of a Sobolev class with restrictions on the derivatives in different metrics}
\date{}
\begin{document}

\maketitle

\newenvironment{Biblio}{%
                  \renewcommand{\refname}{\footnotesize REFERENCES}%
                  \begin{thebibliography}{99}%
                  \itemsep -0.2em}%
                  {\end{thebibliography}}

\def\inff{\mathop{\smash\inf\vphantom\sup}}
\renewcommand{\le}{\leqslant}
\renewcommand{\ge}{\geqslant}
\newcommand{\sgn}{\mathrm {sgn}\,}
\newcommand{\inter}{\mathrm {int}\,}
\newcommand{\dist}{\mathrm {dist}}
\newcommand{\supp}{\mathrm {supp}\,}
\newcommand{\R}{\mathbb{R}}

\newcommand{\Z}{\mathbb{Z}}
\newcommand{\N}{\mathbb{N}}
\newcommand{\Q}{\mathbb{Q}}
\theoremstyle{plain}
\newtheorem{Trm}{Theorem}
\newtheorem{trma}{Theorem}

\newtheorem{Def}{Definition}
\newtheorem{Cor}{Corollary}
\newtheorem{Lem}{Lemma}
\newtheorem{Rem}{Remark}
\newtheorem{Sta}{Proposition}
\newtheorem{Sup}{Assumption}
\newtheorem{Supp}{Assumption}
\newtheorem{Not}{Notation}
\newtheorem{Exa}{Example}
\renewcommand{\proofname}{\bf Proof}
\renewcommand{\thetrma}{\Alph{trma}}
\renewcommand{\theSupp}{\Alph{Supp}}

\begin{abstract}
In this paper, we obtain order estimates for the Kolmogorov widths of periodic Sobolev classes with restictions on derivatives of order $r_j$ with respect to $j$th variable in metrics $L_{p_j}$ ($1\le j\le d$). 
\end{abstract}


\section{Introduction}

In this paper, we consider the problem of the Kolmogorov widths of a periodic Sobolev class on a $d$-dimensional torus $\mathbb{T}^d$ with the following conditions on partial derivatives:
\begin{align}
\label{norm_cond1} \left\| \frac{\partial^{r_j}}{\partial x^{r_j}}f\right\|_{L_{p_j}(\mathbb{T}^d)} \le 1, \quad j=1, \, \dots, \, d. 
\end{align}
This function class is an example of an intersection of several Sobolev classes with a restriction on one of the partial derivatives \cite{galeev_emb, galeev1, galeev2}. Such function classes on $\R^d$ were studied in \cite[\S 6]{algervik}, where a sufficient condition for an embedding into a Lorentz space was obtained. The partial case $p_1=\dots = p_d$ was considered in \cite{kolyada}. Oleinik in \cite{oleinik_vl} considered anisotropic Sobolev classes on a domain $\Omega \subset \R^d$ with restrictions on $\left\| \frac{\partial^{r_j}}{\partial x^{r_j}}f\right\|_{L_{p_j}(\Omega)}$ ($j = 1, \, \dots, \, d$) and on a norm of $f$ in the space $L_{p_0}$ with a special weight. Under the assumption that
\begin{align}
\label{sum_g_0}
1-\sum \limits _{k=1}^d \frac{1}{p_kr_k}>0, \quad q\ge \max _{1\le k\le d} p_k,
\end{align}
embedding theorems into a weighted $L_q$-space and some estimates for the Kolmogorov widths were obtained (in general, these estimates were not sharp with respect to order). Here we will consider non-weighted case for periodic classes, but we do not suppose that the condition \eqref{sum_g_0} holds, and the estimates of the widths will be sharp with respect to order.

Also notice that the class of functions with conditions \eqref{norm_cond1} is an example of anisotropic Sobolev classes. In \cite{besov_iljin_nik, galeev_emb, besov_littlewood, akishev, akishev1, akishev2}, anisotropic Sobolev classes of another type were considered; they were defined by conditions on the derivarives in a mixed norm.

\begin{Def}
Let $X$ be a normed space, $M\subset X$, and let $n\in
\Z_+$. The Kolmogorov $n$-widths of the set $M$ in the space $X$ is defined by
$$
d_n(M, \, X) = \inf _{L\in {\cal L}_n(X)} \sup _{x\in M} \inf
_{y\in L} \|x-y\|;
$$
here ${\cal L}_n(X)$ is a family of all subspaces in $X$
of dimention at most $n$.
\end{Def}

The problem of estimating the widths of finite-dimensional balls and Sobolev classes was studied in \cite{pietsch1, stesin, bib_kashin, gluskin1, bib_gluskin, garn_glus, galeev2}. For details, see \cite{itogi_nt, kniga_pinkusa}.

In \cite{galeev1, galeev2, galeev85, galeev87}, the problem of estimating the widths of a periodic Sobolev class on $\mathbb{T}^d:= [0, \, 2\pi]^d$ with a restriction on the $L_p$-norm of some mixed partial Weyl derivative and of an intersection of such periodic Sobolev classes on $\mathbb{T}^1$ was considered; see also \cite{galeev4}. In \cite{vas_sob}, the result of the paper \cite{galeev2} about the widths of an intersection of Sobolev classes on $\mathbb{T}^1$ was generalized to the case of ``small smoothness'' for $q>2$, except some ``limiting'' cases.

Recall the definition of the Weyl derivative of a periodic function (see, e.g., \cite[Chapter 2, \S 2]{itogi_nt}). Let $d\in \N$, $d\ge 2$, $\mathbb{T}^d= [0, \, 2\pi]^d$. By ${\cal S}'(\mathbb{T}^d)$, we denote the space of distributions on $\mathbb{T}^d$ (as the space of test-functions, we take the set of infinitely smooth periodic functions). Given a distribution $f\in {\cal S}'(\mathbb{T}^d)$, we write its Fourier series: $f = \sum \limits _{\overline{k}\in \Z^d} c_{\overline{k}}(f) e^{i(\overline{k}, \, \cdot)}$, where the series converges with respect to the topology of the space ${\cal S}'(\mathbb{T}^d)$; here $(\cdot, \, \cdot)$ is the standard inner product on $\R^d$. We denote $$\mathaccent'27 \Z^d = \{(k_1, \, k_2, \, \dots, \, k_d)\in \Z^d:\; k_1k_2\dots k_d\ne 0\},$$
$$
\mathaccent'27{\cal S}'(\mathbb{T}^d) = \Bigl\{ f\in {\cal S}'(\mathbb{T}^d):\; f=\sum \limits _{\overline{k}\in \mathaccent'27\Z^d} c_{\overline{k}}(f) e^{i(\overline{k}, \, \cdot)}\Bigr\}.
$$

Let $r_j>0$, $1\le j\le d$. The Weyl derivative of order $r_j$ with respect to the variable $x_j$ of a distribution $f\in \mathaccent'27{\cal S}'(\mathbb{T}^d)$ is defined by
$$
\partial_j^{r_j}f:=\frac{\partial^{r_j}f}{\partial x_j^{r_j}} := \sum \limits _{\overline{k}\in \mathaccent'27\Z^d} c_{\overline{k}}(f) (ik_j)^{r_j} e^{i(\overline{k}, \, \cdot)},
$$
where $(ik_j)^{r_j}=|k_j|^{r_j}e^{{\rm sgn}\,k_j\cdot i\pi r_j/2}$.

Let $1<q<\infty$, $1<p_j<\infty$, $r_j>0$, $1\le j\le d$, $\overline{p}=(p_1, \, \dots, \, p_d)$, $\overline{r} = (r_1, \, \dots, \, r_d)$. We set
$$
W^{\overline{r}}_{\overline{p}}(\mathbb{T}^d) = \{f\in \mathaccent'27{\cal S}'(\mathbb{T}^d):\; \|\partial _j^{r_j}f\|_{L_{p_j}(\mathbb{T}^d)}\le 1, \; 1\le j\le d\}.
$$

In this paper, we obtain order estimates for the Kolmogorov widths of the class $W^{\overline{r}}_{\overline{p}}(\mathbb{T}^d)$ in the space $L_q(\mathbb{T}^d)$.

Given $\overline{a}=(a_1, \, \dots, \, a_d) \in \R^d$, we denote by $\langle \overline{a}\rangle$ the harmonic mean of $a_1, \, \dots, \, a_d$:
$$
\langle \overline{a}\rangle = \frac{d}{\frac{1}{a_1}+\dots + \frac{1}{a_d}}.
$$
Given $\overline{a}=(a_1, \, \dots, \, a_d)$, $\overline{b}=(b_1, \, \dots, \, b_d)\in \R^d$, we write $\overline{a}\circ\overline{b}=(a_1b_1, \, \dots, \, a_db_d)$.

From \cite[Theorem 1]{galeev_emb} it follows that the set $W^{\overline{r}}_{\overline{p}}(\mathbb{T}^d)$ is bounded in $L_q(\mathbb{T}^d)$ if and only if
$$
\frac{\langle \overline{r}\rangle}{d}+ \frac 1q - \frac{\langle \overline{r}\rangle}{\langle \overline{p}\circ\overline{r}\rangle} \ge 0.
$$
If we replace the inequality by the strict one, the embedding will be compact (see \cite[Theorem 5]{galeev_emb}).

Now we formulate the main results of the paper.

First we consider the case when some additional restrictions on the parameters hold. Then (except some ``limiting'' relations on the parameters) the estimates for the widths will be written explicitly.

Let $X$, $Y$ be sets, and let $f_1$, $f_2:\ X\times Y\rightarrow \mathbb{R}_+$.
We denote $f_1(x, \, y)\underset{y}{\lesssim} f_2(x, \, y)$ (or
$f_2(x, \, y)\underset{y}{\gtrsim} f_1(x, \, y)$) if for each 
$y\in Y$ there is $c(y)>0$ such that $f_1(x, \, y)\le
c(y)f_2(x, \, y)$ for all $x\in X$; $f_1(x, \,
y)\underset{y}{\asymp} f_2(x, \, y)$ if $f_1(x, \, y)
\underset{y}{\lesssim} f_2(x, \, y)$ and $f_2(x, \,
y)\underset{y}{\lesssim} f_1(x, \, y)$.

\begin{Trm}
\label{main1} Let $d\in \N$, $d\ge 2$, $1<q<\infty$, $1<p_j<\infty$, $r_j>0$, $j=1, \, \dots, \, d$, and let $\frac{\langle \overline{r}\rangle}{d}+ \frac 1q - \frac{\langle \overline{r}\rangle}{\langle \overline{p}\circ\overline{r}\rangle} > 0$. Suppose that
\begin{align}
\label{dop_usl} \sum \limits _{i=1}^d \frac{1}{r_i}\left(\frac{1}{p_i}-\frac{1}{p_j}\right)<1, \quad j=1, \, \dots, \, d.
\end{align}
\begin{enumerate}
\item Let $p_j\ge q$, $j=1, \, \dots, \, d$. Then
$$
d_n(W^{\overline{r}}_{\overline{p}}(\mathbb{T}^d), \, L_q(\mathbb{T}^d)) \underset{\overline{r}, \overline{p}, q, d}{\asymp} n^{-\langle \overline{r}\rangle /d}.
$$
\item Let $1<q\le 2$.
\begin{enumerate}
\item If $p_i\le q$ for all $i=1, \, \dots, \, d$, then
$$
d_n(W^{\overline{r}}_{\overline{p}}(\mathbb{T}^d), \, L_q(\mathbb{T}^d)) \underset{\overline{r}, \overline{p}, q, d}{\asymp} n^{-\langle \overline{r}\rangle /d-1/q+\langle \overline{r}\rangle/\langle \overline{p} \circ\overline{r}\rangle}.
$$

\item Let $\{i\in \overline{1, \, d}:\; p_i>q\}\ne \varnothing$ and $\{i\in \overline{1, \, d}:\; p_i<q\}\ne \varnothing$; in addition, let $\langle \overline{r}\rangle/\langle \overline{p} \circ\overline{r}\rangle \ne 1/q$. Then
$$
d_n(W^{\overline{r}}_{\overline{p}}(\mathbb{T}^d), \, L_q(\mathbb{T}^d)) \underset{\overline{r}, \overline{p}, q, d}{\asymp} n^{-\min \{\langle \overline{r}\rangle /d, \, \langle \overline{r}\rangle /d+1/q-\langle \overline{r}\rangle/\langle \overline{p} \circ\overline{r}\rangle \}}.
$$
\end{enumerate}
\item Let $2<q<\infty$; we suppose that there is $i\in \{1, \, \dots, \, d\}$ such that $p_i<q$.
We denote $\theta_1 = \frac{\langle \overline{r}\rangle}{d}$, $\theta_2= \frac{\langle \overline{r}\rangle}{d}+ \frac 12 - \frac{\langle \overline{r}\rangle}{\langle \overline{p}\circ\overline{r}\rangle}$,
$\theta_3 = \frac q2\Bigl( \frac{\langle \overline{r}\rangle}{d} + \frac 1q - \frac{\langle \overline{r}\rangle}{\langle \overline{p}\circ \overline{r}\rangle}\Bigr)$.
\begin{enumerate}
\item Let $p_i\le 2$, $1\le i\le d$. Suppose that $\theta_2\ne \theta_3$. Then
$$
d_n(W^{\overline{r}}_{\overline{p}}(\mathbb{T}^d), \, L_q(\mathbb{T}^d)) \underset{\overline{r}, \overline{p}, q, d}{\asymp} n^{-\min \{\theta_2, \, \theta_3\}}.
$$
\item Let $p_i\ge 2$, $1\le i\le d$. Suppose that $\theta_1\ne \theta_3$. Then 
$$
d_n(W^{\overline{r}}_{\overline{p}}(\mathbb{T}^d), \, L_q(\mathbb{T}^d)) \underset{\overline{r}, \overline{p}, q, d}{\asymp} n^{-\min \{\theta_1, \, \theta_3\}}.
$$
\item Let $\{i\in \overline{1, \, d}:\; p_i<2\}\ne \varnothing$, $\{i\in \overline{1, \, d}:\; p_i>2\}\ne \varnothing$. Suppose that there exists $j_*\in \{1, \, 2, \, 3\}$ such that $\theta_{j_*}<\min _{j\ne j_*} \theta_j$. Then
$$
d_n(W^{\overline{r}}_{\overline{p}}(\mathbb{T}^d), \, L_q(\mathbb{T}^d)) \underset{\overline{r}, \overline{p}, q, d}{\asymp} n^{-\theta_{j_*}}.
$$
\end{enumerate}
\end{enumerate}
\end{Trm}

\begin{Rem}
\label{rem1}
We will also prove that, for $p_j\ge q$, $1\le j\le d$, if \eqref{dop_usl} fails, the order estimate from Theorem \ref{main1} holds as well.
\end{Rem}

Now we consider the general case, when (\ref{dop_usl}) may fail.

In what follows, we set $\max \varnothing:=-\infty$.

\begin{Trm}
\label{main} Let $d\in \N$, $d\ge 2$, $1<q<\infty$, $1<p_j<\infty$, $r_j>0$, $j=1, \, \dots, \, d$, and let $\frac{\langle \overline{r}\rangle}{d}+ \frac 1q - \frac{\langle \overline{r}\rangle}{\langle \overline{p}\circ\overline{r}\rangle} > 0$.
\begin{enumerate}
\item Let $1<q\le 2$. We denote
\begin{align}
\label{ij_def}
\begin{array}{c}
I_0=\{j\in \overline{1, \, d}:\; p_j\ge q\}, \quad J_0=\{j\in \overline{1, \, d}:\; p_j\le q\}, \\
I_0'=\{j\in \overline{1, \, d}:\; p_j> q\}, \quad J_0'=\{j\in \overline{1, \, d}:\; p_j< q\};
\end{array}
\end{align}
the numbers $\lambda_{i,j}\in [0, \, 1]$ are defined by
\begin{align}
\label{1qlijql2}
\frac 1q =\frac{1-\lambda_{i,j}}{p_i} + \frac{\lambda_{i,j}}{p_j}, \quad i\in I_0', \; j\in J_0'.
\end{align}
Given $\alpha_1, \, \dots, \, \alpha_d\in \R$, we set
$$
h_1(\alpha_1, \, \dots, \, \alpha_d) = \max _{j\in I_0} r_j\alpha_j,
$$
$$
h_2(\alpha_1, \, \dots, \, \alpha_d) = \max_{j\in J_0}(r_j\alpha_j-1/p_j+1/q),
$$
$$
h_3(\alpha_1, \, \dots, \, \alpha_d) = \max _{i\in I_0', \, j\in J_0'} ((1-\lambda_{i,j})r_i\alpha_i+\lambda_{i,j}r_j\alpha_j),
$$
\begin{align}
\label{h_def_00}
h(\alpha_1, \, \dots, \, \alpha_d)= \max _{1\le j\le 3} h_j(\alpha_1, \, \dots, \, \alpha_d).
\end{align}
Suppose that the minimum point $(\hat\alpha_1, \, \dots, \, \hat\alpha_d)$ of the function $h$ on the set
\begin{align}
\label{d_set_def_00}
D = \{(\alpha_1, \, \dots, \, \alpha_d)\in \R^{d}:\; \alpha_1\ge 0, \, \dots, \, \alpha_d\ge 0, \; \alpha_1+\dots+\alpha_d=1\}
\end{align}
is unique. Then
$$
d_n(W^{\overline{r}}_{\overline{p}}(\mathbb{T}^d), \, L_q(\mathbb{T}^d)) \underset{\overline{r}, \, \overline{p}, \, q, \, d}{\asymp} n^{-h(\hat \alpha_1, \, \dots, \, \hat \alpha_d)}.
$$

\item Let $2<q<\infty$. We set
\begin{align}
\label{ijk_def}
\begin{array}{c}
I=\{j\in \overline{1, \, d}:\; p_j\ge q\}, \quad J=\{j\in \overline{1, \, d}:\; 2\le p_j\le q\},
\\
K=\{j\in \overline{1, \, d}:\; p_j\le 2\}, \\
I'=\{j\in \overline{1, \, d}:\; p_j> q\}, \quad J'=\{j\in \overline{1, \, d}:\; 2< p_j< q\},
\\
K'=\{j\in \overline{1, \, d}:\; p_j< 2\};
\end{array}
\end{align}
the numbers $\lambda_{i,j}\in [0, \, 1]$ and $\mu_{i,j}\in [0, \, 1]$ are defined by
\begin{align}
\label{lam_ij_mu_ij_def}
\begin{array}{c}
\frac 1q =\frac{1-\lambda_{i,j}}{p_i} + \frac{\lambda_{i,j}}{p_j}, \quad i\in I', \; j\in J'\cup K; \\ \frac 12 =\frac{1-\mu_{i,j}}{p_i} + \frac{\mu_{i,j}}{p_j}, \quad i\in I\cup J', \; j\in K'.
\end{array}
\end{align}
Given $\alpha_1, \, \dots, \, \alpha_d\in \R$, $s\in \R$, we write
$$
\tilde h_1(\alpha_1, \, \dots, \, \alpha_d, \, s) = \max _{j\in I} r_j\alpha_j,
$$
$$
\tilde h_2(\alpha_1, \, \dots, \, \alpha_d, \, s) = \max _{j\in J} \left(r_j\alpha_j -\frac 12\cdot \frac{1/p_j-1/q}{1/2-1/q}(s-1)\right),
$$
$$
\tilde h_3(\alpha_1, \, \dots, \, \alpha_d, \, s) = \max_{j\in K}(r_j\alpha_j-s/p_j+1/2),
$$
$$
\tilde h_4(\alpha_1, \, \dots, \, \alpha_d, \, s) = \max _{i\in I', \, j\in J'\cup K} ((1-\lambda_{i,j})r_i\alpha_i+\lambda_{i,j}r_j\alpha_j),
$$
$$
\tilde h_5(\alpha_1, \, \dots, \, \alpha_d, \, s) = \max _{i\in I\cup J', \, j\in K'} ((1-\mu_{i,j})r_i\alpha_i+\mu_{i,j}r_j\alpha_j-s/2+1/2),
$$
$$
\tilde h(\alpha_1, \, \dots, \, \alpha_d, \, s)= \max _{1\le j\le 5} \tilde h_j(\alpha_1, \, \dots, \, \alpha_d, \, s).
$$
Suppose that the minimum point $(\hat \alpha_1, \, \dots, \, \hat \alpha_d, \, \hat s)$ of the function $\tilde h$ on the set
$$
\tilde D = \{(\alpha_1, \, \dots, \, \alpha_d, \, s)\in \R^{d+1}:\; 1\le s \le q/2,$$$$\alpha_1\ge 0, \, \dots, \, \alpha_d\ge 0, \; \alpha_1+\dots+\alpha_d=s\}
$$
is unique. Then
$$
d_n(W^{\overline{r}}_{\overline{p}}(\mathbb{T}^d), \, L_q(\mathbb{T}^d)) \underset{\overline{r}, \, \overline{p}, \, q, \, d}{\asymp} n^{-\tilde h(\hat \alpha_1, \, \dots, \, \hat \alpha_d, \, \hat s)}.
$$
\end{enumerate}
\end{Trm}

Now we consider the case when the condition (\ref{dop_usl}) fails.

\begin{Trm}
\label{not_emb} Let $d\in \N$, $d\ge 2$, $r_k>0$, $1<p_k\le q<\infty$ for all $k=1, \, \dots, \, d$. Suppose that there is $j\in \{1, \, \dots, \, d\}$ such that $\sum \limits _{i=1}^d \frac{1}{r_i}\left(\frac{1}{p_i}-\frac{1}{p_j}\right)\ge 1$. Then $W^{\overline{r}}_{\overline{p}}(\mathbb{T}^d)$ is not compactly embedded into $L_q(\mathbb{T}^d)$.
\end{Trm}

For $d=2$, when the condition (\ref{dop_usl}) fails, the orders of the widths are written explicitly. Let, without loss of generality, $p_1>p_2$. Then $r_2\le \frac{1}{p_2} -\frac{1}{p_1}$. By Remark \ref{rem1} and Theorem \ref{not_emb}, it suffices to consider the case $p_2<q<p_1$.

\begin{Trm}
\label{main2} Let $1<p_2<q<p_1<\infty$, $r_1>0$, $r_2>0$. Suppose that 
\begin{align}
\label{r2lp2p1}
r_2\le \frac{1}{p_2} -\frac{1}{p_1}.
\end{align}
\begin{enumerate}
\item Let $1<q\le 2$.
We define the number $\lambda\in (0, \, 1)$ by the equation $\frac 1q= \frac{1-\lambda}{p_1} + \frac{\lambda}{p_2}$. Suppose that $\frac{\langle \overline{r}\rangle}{2} \ne \lambda r_2$. Then
\begin{align}
\label{dn_wrp_r2}
d_n(W^{\overline{r}}_{\overline{p}}(\mathbb{T}^2), \, L_q(\mathbb{T}^2)) \underset{\overline{r}, \, \overline{p}, \, q}{\asymp} n^{-\min \{\langle \overline{r}\rangle/2, \, \lambda r_2\}}.
\end{align}
\item Let $2<q<\infty$. We define the number $\lambda\in (0, \, 1)$ by the equation $\frac 1q= \frac{1-\lambda}{p_1} +\frac{\lambda}{p_2}$, and the number $\hat s$, by the equation $\hat s\left(1-\frac{r_2(1-2/q)}{1/p_2-1/p_1}\right)=1$.
\begin{enumerate}
\item Let $p_2\ge 2$. We set $\theta_1 = \frac{\langle \overline{r}\rangle}{2}$, $\theta_2 = \hat s\lambda r_2$. Suppose that $\theta_1\ne \theta_2$. Then 
$$
d_n(W^{\overline{r}}_{\overline{p}}(\mathbb{T}^2), \, L_q(\mathbb{T}^2)) \underset{\overline{r}, \, \overline{p}, \, q}{\asymp} n^{-\min \{\theta_1, \, \theta_2\}}.
$$
\item Let $p_2<2$. We define the number $\mu\in (0, \, 1)$ by the equation $\frac 12= \frac{1-\mu}{p_1} +\frac{\mu}{p_2}$. Let $\theta_1 = \frac{\langle \overline{r}\rangle}{2}$, $\theta_2 = \hat s\lambda r_2$, $\theta_3 = \mu r_2$. Suppose that there exists $j_*\in \{1, \, 2, \, 3\}$ such that
\begin{align}
\label{theta_j_st_l_min}
\theta_{j_*} = \min _{j\ne j_*} \theta_j.
\end{align}
Then
$$
d_n(W^{\overline{r}}_{\overline{p}}(\mathbb{T}^2), \, L_q(\mathbb{T}^2)) \underset{\overline{r}, \, \overline{p}, \, q}{\asymp} n^{-\theta_{j_*}}.
$$
\end{enumerate}
\end{enumerate}
\end{Trm}

\section{Auxiliary assertions}

Let $\overline{m}=(m_1, \, \dots, \, m_d)\in \N^d$. We write 
\begin{align}
\label{m_sum_def}
m = m_1+\dots + m_d,
\end{align}
$$\square_{\overline{m}}=\{k\in \Z^d:\; 2^{m_j-1}\le |k_j|<2^{m_j}, \; 1\le j\le d\},$$ ${\cal T}_{\overline{m}} = {\rm span} \{e^{i(\overline{k},\, \cdot)}\}_{\overline{k}\in \square_{\overline{m}}}$. Given $f(\cdot)= \sum \limits _{\overline{k}\in \mathaccent'27\Z^d} c_{\overline{k}}(f)e^{i(\overline{k},\, \cdot)}$, we set $$\delta_{\overline{m}} f(\cdot)= \sum \limits _{\overline{k}\in \square_{\overline{m}}} c_{\overline{k}}(f)e^{i(\overline{k},\, \cdot)}.$$

For each $f\in \mathaccent'27{\cal S}'(\mathbb{T}^d)$, we write
$$
Pf(t)= \left(\sum \limits _{\overline{m}\in \N^d}|\delta_{\overline{m}}f(t)|^2\right)^{1/2}.
$$

\begin{trma} \label{lit_pal}
{\rm (Littlewood--Paley theorem; see \cite[section 1.5.2]{nikolski_sm}, \cite[Ch. 2, section 2.3, Theorem 15]{itogi_nt}, \cite[Ch. III, section 15.2]{besov_iljin_nik}, \cite{besov_littlewood}.)} 
Let $1<q<\infty$. Then $f\in L_q(\mathbb{T}^d)$ if and only if $Pf\in L_q(\mathbb{T}^d)$; in addition, $\|f\|_{L_q(\mathbb{T}^d)} \underset{q,d}{\asymp} \|Pf\|_{L_q(\mathbb{T}^d)}$.
\end{trma}

\begin{trma}
\label{dif_r} 
Let $1<p_j<\infty$, $r_j\in \R$. Then $$\|\partial _j^{r_j} f\|_{L_{p_j}(\mathbb{T}^d)} \underset{\overline{p},\overline{r},d}{\asymp} 2^{m_jr_j} \|x\|_{L_{p_j}(\mathbb{T}^d)}$$ for each $f\in {\cal T}_{\overline{m}}$.
\end{trma}
This estimate follows from Marcinkiewicz multiplier theorem \cite[section 1.5.3]{nikolski_sm}, \cite[Ch. III, section 15.3]{besov_iljin_nik}; see also \cite[Ch. 2, section 2.3, Theorem 18]{itogi_nt} for $r_j \ge 0$.

\begin{trma}
\label{fin_dim_isom} {\rm (see \cite{galeev2}, Theorem B; \cite[Vol. 2, Ch. X, Theorem 7.5]{zigmund}).} There is an isomorphism $A:{\cal T}_{\overline{m}} \rightarrow \R^{2^m}$ such that for all $q\in (1, \, \infty)$, $f\in {\cal T}_{\overline{m}}$ the estimate $$\|f\|_{L_q(\mathbb{T}^d)} \underset{q, \, d}{\asymp} 2^{-m/q} \|Ax\|_{l_q^{2^m}}$$ holds.
\end{trma}

The estimates for the widths $d_n(B_p^N, \, l_q^N)$ were obtained by Pietsch, Stesin, Kashin, Gluskin and Garnaev \cite{gluskin1, bib_gluskin, garn_glus, bib_kashin, pietsch1, stesin}. Let us formulate the results for the cases that will be considered below.

\begin{trma}
\label{glus} {\rm (see \cite{bib_gluskin}).} Let $1\le p\le q<\infty$,
$0\le n\le N/2$.
\begin{enumerate}
\item Let $1\le q\le 2$. Then $d_n(B_p^N, \, l_q^N) \asymp
1$.

\item Let $2<q<\infty$, $\omega_{pq} =\min \left\{1, \,
\frac{1/p-1/q}{1/2-1/q}\right\}$. Then
$$
d_n(B_p^N, \, l_q^N) \underset{q}{\asymp} \min \{1, \,
n^{-1/2}N^{1/q}\} ^{\omega_{pq}}.
$$
\end{enumerate}
\end{trma}

\begin{trma}
\label{p_s} {\rm (see \cite{pietsch1, stesin}).} Let $1\le q\le p\le
\infty$, $0\le n\le N$. Then
$$
d_n(B_p^N, \, l_q^N) = (N-n)^{1/q-1/p}.
$$
\end{trma}

Order estimates for the Kolmogorov $n$-widths of an intersection of a family of $N$-dimensional balls were obtained in \cite{galeev1} for $N=2n$ and in \cite{vas_ball_inters} for $N\ge 2n$. In \cite{vas_ball_inters}, the explicit formula for the order estimate was written under an additional condition when no ball of the family contains another one. In \cite[Proposition 1]{vas_sob} the estimate was written in a general case for a finite family of balls. Let us formulate this result.

\begin{trma}
\label{fin_dim_inters} {\rm (see \cite[Proposition 1]{vas_sob}).}
Let $A$ be a finite nonempty set, $1\le p_\alpha\le \infty$, $\nu_\alpha>0$, $\alpha \in A$,
\begin{align}
\label{m_set_def}
M_0 = \cap _{\alpha \in A} \nu _\alpha B_{p_\alpha}^N, 
\end{align}
$N\ge 2n$. We define the numbers $\lambda_{\alpha,\beta}$ and $\tilde \lambda_{\alpha,\beta}$ by the equations
\begin{align}
\label{q_l_ab}
\frac{1}{q}= \frac{1-\lambda_{\alpha,\beta}}{p_\alpha} + \frac{\lambda_{\alpha,\beta}}{p_\beta}, \quad p_\alpha > q, \; p_\beta < q,
\end{align}
\begin{align}
\label{2_tl_ab}
\frac{1}{2}= \frac{1-\tilde\lambda_{\alpha,\beta}}{p_\alpha} + \frac{\tilde\lambda_{\alpha,\beta}}{p_\beta}, \quad p_\alpha > 2, \; p_\beta < 2.
\end{align}
Then, for $q\le 2$,
\begin{align}
\label{dnm0}
d_n(M_0, \, l_q^N) \asymp \min \left\{ \min _{\alpha \in A}d_n(\nu_\alpha B_{p_\alpha}^N, \, l_q^N), \, \min _{p_\alpha>q, \, p_\beta< q} \nu_\alpha ^{1-\lambda_{\alpha,\beta}}\nu_\beta^{\lambda_{\alpha,\beta}}\right\};
\end{align}
for $q>2$,
\begin{align}
\label{dnm0q2}
\begin{array}{c}
d_n(M_0, \, l_q^N) \underset{q}{\asymp} \min \left\{ \min _{\alpha \in A}d_n(\nu_\alpha B_{p_\alpha}^N, \, l_q^N), \, \min _{p_\alpha>q, \, p_\beta< q} \nu_\alpha ^{1-\lambda_{\alpha,\beta}}\nu_\beta^{\lambda_{\alpha,\beta}},\right. \\ \left. \min _{p_\alpha> 2, \, p_\beta< 2} \nu_\alpha ^{1-\tilde\lambda_{\alpha,\beta}}\nu_\beta^{\tilde\lambda_{\alpha,\beta}}d_n(B_2^N, \, l_q^N)\right\}.
\end{array}
\end{align}
\end{trma}
In the formulation of Proposition 1 from \cite{vas_sob} the inequalities $p_\alpha> q$, $p_\beta< q$, $p_\alpha> 2$, $p_\beta< 2$ from \eqref{dnm0}, \eqref{dnm0q2} were replaced by nonstrict ones; the estimate is the same.

Given $k\in \{1,\, \dots, \, N\}$, we define the sets $V_k\subset \R^N$ by 
$$
V_k={\rm conv}\, \{(\varepsilon_1 \hat{x}_{\sigma(1)}, \, \dots,
\, \varepsilon_N \hat{x}_{\sigma(N)}):\; \varepsilon_j=\pm 1,
\;1\le j\le N, \; \sigma \in S_N\},
$$
where $\hat{x}_j=1$ for $1\le j\le k$, $\hat{x}_j=0$ for $k+1\le j\le
N$, $S_N$ is the permutation group of $N$ elements. Notice that $V_1
= B_1^N$, $V_N = B_\infty^N$.

For $2\le q<\infty$, the lower estimates for $d_n(V_k, \, l_q^N)$ were obtained by Gluskin \cite{gluskin1}.

\begin{trma}
\label{gl_q_g_2} {\rm \cite{gluskin1}} Let $2\le q<\infty$, $1\le k\le N$. Then
\begin{align}
\label{kq1} d_n(V_k, \, l_q^N) \underset{q}{\gtrsim} \begin{cases} k^{1/q} & \text{for }n\le 
\min \{N^{\frac 2q}k^{1 -\frac 2q}, \, N/2\}, \\ k^{1/2}n^{-1/2}N^{1/q} & \text{for }N^{\frac 2q}k^{1 -\frac 2q}
\le n\le N/2. \end{cases}
\end{align}
\end{trma}

The following result was obtained by Gluskin \cite{bib_glus_3} (with the constant in order inequality depending on $q$), Malykhin and Rjutin \cite{mal_rjut} (with the constant independent of $q$). In \cite[p. 39]{bib_glus_3} it is noticed that Galeev obtained the equality $d_n(V_k, \, l_1^N) =\min\{k, \, N-n\}$.
\begin{trma}
 {\rm \cite{bib_glus_3, mal_rjut}.} Let $1\le q\le 2$, $n\le
N/2$. Then 
\begin{align}
\label{gl_q_l_2}
d_n(V_k, \, l_q^N) \gtrsim k^{1/q}.
\end{align}
\end{trma}

\section{On estimates of the widths of an intersection of finite-dimensional balls}

In this section, we clarify the formulas for the estimates of the widths from Theorem \ref{fin_dim_inters}.

Let, first, $1\le q\le 2$. Then from (\ref{dnm0}) and Theorems \ref{glus}, \ref{p_s} it follows that for $n\le N/2$
\begin{align}
\label{dnm01}
d_n(M_0, \, l_q^N) \asymp \min \{\min _{p_\alpha \ge q} \nu_\alpha N^{1/q-1/p_\alpha}, \, \min _{p_\alpha \le q} \nu_\alpha, \, \min _{p_\alpha> q, \, p_\beta< q} \nu_\alpha ^{1-\lambda_{\alpha,\beta}}\nu_\beta^{\lambda_{\alpha,\beta}}\}.
\end{align}

\begin{Lem}
\label{expl1} Let $1\le q\le 2$, $p_\alpha\ne q$ for all $\alpha\in A$, $n \le N/2$, and let the set $M_0$ be defined by formula \eqref{m_set_def}.
\begin{enumerate}
\item Let $p_{\alpha_*} < q$, $\nu_{\alpha_*} \le \nu_\beta$ for each $\beta \in A$. Then $d_n(M_0, \, l_q^N) \asymp \nu_{\alpha_*}$.

\item Let $p_{\alpha_*} > q$, $\nu_{\alpha_*} N^{1/p_\beta -1/p_{\alpha_*}}\le \nu_\beta$ for each $\beta \in A$. Then $d_n(M_0, \, l_q^N) \asymp \nu_{\alpha_*} N^{1/q-1/p_{\alpha_*}}$.

\item Let $p_{\alpha_*} > q$, $p_{\beta_*}< q$, and let
\begin{align}
\label{nu_ab}
\begin{array}{c}
\nu_{\alpha_*} ^{1-\lambda_{\alpha_*,\beta_*}}\nu_{\beta_*}^{\lambda_{\alpha_*,\beta_*}} \le \nu_{\alpha_*} ^{1-\lambda_{\alpha_*,\gamma}}\nu_\gamma^{\lambda_{\alpha_*,\gamma}}, \quad \gamma \in A, \; p_\gamma < q,
\\
\nu_{\alpha_*} ^{1-\lambda_{\alpha_*,\beta_*}}\nu_{\beta_*}^{\lambda_{\alpha_*,\beta_*}} \le \nu_\gamma ^{1-\lambda_{\gamma,\beta_*}}\nu_{\beta_*}^{\lambda_{\gamma,\beta_*}}, \quad \gamma \in A, \; p_\gamma > q,
\end{array}
\end{align}
\begin{align}
\label{nu_ab1}
\nu_{\alpha_*}  \le \nu_{\beta_*}, \quad \nu_{\alpha_*}\ge \nu_{\beta_*} N^{1/p_{\alpha_*}-1/p_{\beta_*}}.
\end{align}
Then
$d_n(M_0, \, l_q^N) \asymp \nu_{\alpha_*} ^{1-\lambda_{\alpha_*,\beta_*}}\nu_{\beta_*}^{\lambda_{\alpha_*,\beta_*}}$.
\end{enumerate}
\end{Lem}
\begin{proof}
By (\ref{dnm01}), it suffices to prove the lower estimates for the widths $d_n(M_0, \, l_q^N)$.

In case 1, we use the inclusion $\nu_{\alpha_*} B^N_1 \subset M_0$ and Theorem \ref{glus}, in case 2, the inclusion $\nu_{\alpha_*} N^{-1/p_{\alpha_*}} B^N_{\infty} \subset M_0$ and Theorem \ref{p_s}.

In case 3, we define the number $l$ by the equation $\frac{\nu_{\alpha_*}}{\nu_{\beta_*}}= l^{1/p_{\alpha_*}-1/p_{\beta_*}}$ and set $k = \lceil l\rceil$. From (\ref{nu_ab1}) it follows that $1\le l\le N$ and, hence, $1\le k\le N$. We show that $\nu_{\alpha_*} ^{1-\lambda_{\alpha_*,\beta_*}}\nu_{\beta_*}^{\lambda_{\alpha_*,\beta_*}} k^{-1/q}V_k \subset 2M_0$; then the lower estimate for $d_n(M_0, \, l_q^N)$ follows from (\ref{gl_q_l_2}). It suffices to check that, for each $\gamma\in A$,
\begin{align}
\label{nu_a_l_nu_g}
\nu_{\alpha_*} ^{1-\lambda_{\alpha_*,\beta_*}}\nu_{\beta_*}^{\lambda_{\alpha_*,\beta_*}} l^{1/p_\gamma-1/q}\le \nu_\gamma;
\end{align}
i.e., $\nu_{\alpha_*} l^{1/p_\gamma-1/p_{\alpha_*}} \le \nu_\gamma$ (see (\ref{q_l_ab}) and the definition of $l$). The last inequality follows from (\ref{nu_ab}); the arguments are the same as in \cite[p. 6]{vas_ball_inters}.
\end{proof}

Now we consider the case $q>2$. From (\ref{dnm0q2}) and Theorems \ref{glus}, \ref{p_s} it follows that, for $N^{2/q}\le n \le \frac{N}{2}$,
\begin{align}
\label{dnm0q21}
\begin{array}{c}
d_n(M_0, \, l_q^N) \underset{q}{\asymp} \min \left\{ \min _{p_\alpha \ge q}\nu_\alpha N^{1/q-1/p_\alpha}, \, \min _{2\le p_\alpha \le q}\nu_\alpha (n^{-1/2}N^{1/q})^{\frac{1/p_\alpha-1/q}{1/2-1/q}},\right. \\ \left.\min _{p_\alpha \le 2} \nu_\alpha n^{-1/2}N^{1/q}, \, \min _{p_\alpha> q, \, p_\beta< q} \nu_\alpha ^{1-\lambda_{\alpha,\beta}}\nu_\beta^{\lambda_{\alpha,\beta}}, \, \min _{p_\alpha> 2, \, p_\beta< 2} \nu_\alpha ^{1-\tilde\lambda_{\alpha,\beta}}\nu_\beta^{\tilde\lambda_{\alpha,\beta}}n^{-1/2}N^{1/q}\right\}.
\end{array}
\end{align}

\begin{Lem}
\label{expl2} Let $2< q<\infty$, $p_\alpha\notin \{2, \, q\}$ for all $\alpha\in A$, $N^{2/q}\le n \le N/2$, and let the set $M_0$ be defined by formula \eqref{m_set_def}.
\begin{enumerate}
\item Let $p_{\alpha_*} < 2$, $\nu_{\alpha_*} \le \nu_\beta$ for all $\beta \in A$. Then $d_n(M_0, \, l_q^N) \underset{q}{\asymp} \nu_{\alpha_*}n^{-1/2}N^{1/q}$.

\item Let $p_{\alpha_*} > q$, $\nu_{\alpha_*} N^{1/p_\beta -1/p_{\alpha_*}}\le \nu_\beta$ for all $\beta \in A$. Then $d_n(M_0, \, l_q^N) \underset{q}{\asymp} \nu_{\alpha_*} N^{1/q-1/p_{\alpha_*}}$.

\item Let $2<p_{\alpha_*}<q$, 
\begin{align}
\label{nuan12n1q}
\nu_{\alpha_*}(n^{1/2}N^{-1/q})^{\frac{1/p_\beta-1/p_{\alpha_*}}{1/2-1/q}} \le \nu_\beta
\end{align}
for all $\beta \in A$. Then 
\begin{align}
\label{dn_m0_j}
d_n(M_0, \, l_q^N) \underset{q}{\asymp} \nu_{\alpha_*} (n^{-1/2}N^{1/q})^{\frac{1/p_{\alpha_*}-1/q}{1/2-1/q}}.
\end{align}

\item Let $p_{\alpha_*} > q$, $p_{\beta_*}< q$, and let \eqref{nu_ab} hold, as well as
\begin{align}
\label{nu_ab1_lem2}
\nu_{\alpha_*}  \le \nu_{\beta_*}(n^{1/2}N^{-1/q})^{\frac{1/p_{\alpha_*}-1/p_{\beta_*}}{1/2-1/q}}, \quad \nu_{\alpha_*}\ge \nu_{\beta_*} N^{1/p_{\alpha_*}-1/p_{\beta_*}}.
\end{align}
Then
$d_n(M_0, \, l_q^N) \underset{q}{\asymp} \nu_{\alpha_*} ^{1-\lambda_{\alpha_*,\beta_*}}\nu_{\beta_*}^{\lambda_{\alpha_*,\beta_*}}$.

\item Let $p_{\alpha_*} > 2$, $p_{\beta_*}< 2$, and let 
\begin{align}
\label{nu_ab_lem2}
\begin{array}{c}
\nu_{\alpha_*} ^{1-\tilde\lambda_{\alpha_*,\beta_*}}\nu_{\beta_*}^{\tilde\lambda_{\alpha_*,\beta_*}} \le \nu_{\alpha_*} ^{1-\tilde\lambda_{\alpha_*,\gamma}}\nu_\gamma^{\tilde\lambda_{\alpha_*,\gamma}}, \quad \gamma \in A, \; p_\gamma < 2,
\\
\nu_{\alpha_*} ^{1-\tilde\lambda_{\alpha_*,\beta_*}}\nu_{\beta_*}^{\tilde\lambda_{\alpha_*,\beta_*}} \le \nu_\gamma ^{1-\tilde\lambda_{\gamma,\beta_*}}\nu_{\beta_*}^{\tilde\lambda_{\gamma,\beta_*}}, \quad \gamma \in A, \; p_\gamma > 2,
\end{array}
\end{align}
\begin{align}
\label{nu_ab2_lem2}
\nu_{\alpha_*}  \le \nu_{\beta_*}, \quad \nu_{\alpha_*}\ge \nu_{\beta_*} (n^{1/2}N^{-1/q})^{\frac{1/p_{\alpha_*}-1/p_{\beta_*}}{1/2-1/q}}.
\end{align}
Then
$d_n(M_0, \, l_q^N) \underset{q}{\asymp} \nu_{\alpha_*} ^{1-\tilde\lambda_{\alpha_*,\beta_*}}\nu_{\beta_*}^{\tilde \lambda_{\alpha_*,\beta_*}}n^{-1/2}N^{1/q}$.
\end{enumerate}
\end{Lem}
\begin{proof}
By (\ref{dnm0q21}), it suffices to prove the lower estimate.

In case 1, we use the inclusion $\nu_{\alpha_*}B_1^N \subset M_0$ and Theorem \ref{glus}, in case 2, the inclusion $\nu_{\alpha_*}N^{-1/p_{\alpha_*}}B_\infty^N \subset M_0$ and Theorem \ref{p_s}.

In case 3, we set $l = (n^{1/2}N^{-1/q})^{\frac{1}{1/2-1/q}}$, $k=\lceil l\rceil$. Then $1\le l\le N$, $1\le k\le N$, $n \le N^{\frac 2q} k^{1-\frac 2q}$. We claim that $\nu_{\alpha_*}k^{-1/p_{\alpha_*}}V_k \subset 2M_0$; then (\ref{dn_m0_j}) follows from (\ref{kq1}). It suffices to check the inequality $\nu_{\alpha_*} l^{1/p_\beta-1/p_{\alpha_*}} \le \nu_\beta$, $\beta\in A$. It follows from (\ref{nuan12n1q}).

In cases 4, 5, we define the number $l$ by the equation 
\begin{align}
\label{l_def_nuab}
\frac{\nu_{\alpha_*}}{\nu_{\beta_*}}= l^{1/p_{\alpha_*}-1/p_{\beta_*}}.
\end{align}

In case 4, we set $k = \lceil l\rceil$. From (\ref{nu_ab1_lem2}) it follows that $(n^{1/2}N^{-1/q})^{\frac{1}{1/2-1/q}}\le l\le N$. Hence $1\le k\le N$ and $n \le N^{\frac 2q} k^{1-\frac 2q}$. We prove that $\nu_{\alpha_*} ^{1-\lambda_{\alpha_*,\beta_*}}\nu_{\beta_*}^{\lambda_{\alpha_*,\beta_*}} k^{-1/q}V_k \subset 2M_0$ and apply (\ref{kq1}). In order to prove the inclusion it suffices to check that (\ref{nu_a_l_nu_g}) holds for each $\gamma\in A$; by (\ref{q_l_ab}) and (\ref{l_def_nuab}), it is equivalent to $\nu_{\alpha_*} l^{1/p_\gamma-1/p_{\alpha_*}} \le \nu_\gamma$. The last inequality follows from (\ref{nu_ab}); the arguments are the same as in \cite[pp. 11--12]{vas_ball_inters}.

In case 5, we set $k = \lfloor l\rfloor$. From (\ref{nu_ab2_lem2}) it follows that $1\le l\le (n^{1/2}N^{-1/q})^{\frac{1}{1/2-1/q}}$. Hence $1\le k\le N$ and $n \ge N^{\frac 2q} k^{1-\frac 2q}$. We show that $\nu_{\alpha_*} ^{1-\tilde\lambda_{\alpha_*,\beta_*}}\nu_{\beta_*}^{\tilde\lambda_{\alpha_*,\beta_*}} k^{-1/2}V_k \subset 2M_0$ and use (\ref{kq1}). To this end, it suffices to check that for each $\gamma\in A$ the inequality $$\nu_{\alpha_*} ^{1-\tilde\lambda_{\alpha_*,\beta_*}}\nu_{\beta_*}^{\tilde\lambda_{\alpha_*,\beta_*}} l^{1/p_\gamma-1/2}\le \nu_\gamma$$ holds; by (\ref{2_tl_ab}) and (\ref{l_def_nuab}), it is equivalent to $\nu_{\alpha_*} l^{1/p_\gamma-1/p_{\alpha_*}} \le \nu_\gamma$. The last relation follows from (\ref{nu_ab_lem2}); the arguments are the same as in \cite[pp. 12--13]{vas_ball_inters}.
\end{proof}

\section{Proof of Theorem \ref{main}}

The following assertions reduce the problem of estimating the widths $W^{\overline{r}}_{\overline{p}}(\mathbb{T}^d)$ to the problem of estimating the widths of intersections of finite-dimensional balls. Recall that for $\overline{m}\in \N^d$ the number $m$ is defined by \eqref{m_sum_def}.

\begin{Lem}
\label{low_est_lem}
Let $n\in \Z_+$. Then
\begin{align}
\label{per_low} d_n(W^{\overline{r}}_{\overline{p}}(\mathbb{T}^d), \, L_q(\mathbb{T}^d)) \underset{\overline{p},q,\overline{r},d}{\gtrsim} d_n\left(\cap _{j=1}^d 2^{-m_jr_j-m/q+m/p_j} B^{2^m}_{p_j}, \, l_q^{2^m}\right), \quad \overline{m}\in \N^d.
\end{align}
\end{Lem}
\begin{proof}
The arguments are the same as in \cite[Theorem 1]{galeev2}: applying Theorems \ref{lit_pal}, \ref{dif_r} and \ref{fin_dim_isom}, we get the following order inequalities:
$$
d_n(W^{\overline{r}}_{\overline{p}}(\mathbb{T}^d), \, L_q(\mathbb{T}^d)) \underset{\overline{p},q,\overline{r},d}{\gtrsim} d_n(W^{\overline{r}}_{\overline{p}}(\mathbb{T}^d)\cap {\cal T}_{\overline{m}}, \, L_q(\mathbb{T}^d)\cap {\cal T}_{\overline{m}}) \underset{\overline{p},q,\overline{r},d}{\gtrsim}$$$$\gtrsim d_n(\cap _{j=1}^d 2^{-m_jr_j-m/q+m/p_j} B^{2^m}_{p_j}, \, l_q^{2^m}).
$$
\end{proof}
\begin{Lem}
\label{up_est_lem_m}
Let $k\in \Z_+$. Then
\begin{align}
\label{per_up_m} d_k(\delta_{\overline{m}}W^{\overline{r}}_{\overline{p}}(\mathbb{T}^d), \, L_q(\mathbb{T}^d)) \underset{\overline{p},q,\overline{r},d}{\lesssim} d_k\left(\cap _{j=1}^d 2^{-m_jr_j-m/q+m/p_j} B^{2^m}_{p_j}, \, l_q^{2^m}\right), \quad \overline{m}\in \N^d.
\end{align}
\end{Lem}
\begin{proof}
The assertion follows from Theorems \ref{lit_pal}, \ref{dif_r} and \ref{fin_dim_isom}.
\end{proof}

In particular, for each function $f\in W^{\overline{r}}_{\overline{p}}(\mathbb{T}^d)$, we have 
\begin{align}
\label{cmdef}
\|\delta_{\overline{m}}f\|_{L_q(\mathbb{T}^d)} \underset{\overline{p},q,\overline{r},d}{\lesssim} d_0\left(\cap _{j=1}^d 2^{-m_jr_j-m/q+m/p_j} B^{2^m}_{p_j}, \, l_q^{2^m}\right) =:C_{\overline{m}}.
\end{align}
Now we apply Theorems \ref{glus}, \ref{p_s} and \ref{fin_dim_inters}. Using the notation (\ref{ij_def}) and (\ref{1qlijql2}), we get that, for $q\le 2$,
$$
C_{\overline{m}}\lesssim \min \Bigl\{ \min _{j\in I_0} 2^{-m_jr_j}, \, \min _{j\in J_0} 2^{-r_jm_j-m/q+m/p_j}, \, \min _{i\in I_0', \, j\in J_0'} 2^{-(1-\lambda_{i,j})r_im_i-\lambda_{i,j}r_jm_j}\Bigr\}.
$$
For $2<q<\infty$ we use (\ref{ijk_def}), (\ref{lam_ij_mu_ij_def}) and get
$$
C_{\overline{m}} \underset{q}{\lesssim} \min \Bigl\{ \min _{j\in I} 2^{-m_jr_j}, \, \min _{j\in J\cup K} 2^{-r_jm_j-m/q+m/p_j}, \, \min _{i\in I', \, j\in J'\cup K} 2^{-(1-\lambda_{i,j})r_im_i-\lambda_{i,j}r_jm_j},
$$
$$
\min _{i\in I\cup J', \, j\in K'} 2^{-(1-\mu_{i,j})r_im_i-\mu_{i,j}r_jm_j -m/q+m/2}\Bigr\} = 
$$
$$
= \min \Bigl\{ \min _{j\in I} 2^{-m_jr_j}, \, \min _{j\in J\cup K} 2^{-r_jm_j-m/q+m/p_j}, \, \min _{i\in I', \, j\in J'\cup K} 2^{-(1-\lambda_{i,j})r_im_i-\lambda_{i,j}r_jm_j}\Bigr\};
$$
the last equality is true since for $i\in I\cup J'$, $j\in K'$
\begin{align}
\label{1tlijriti}
\begin{array}{c}
(1-\mu_{i,j})r_im_i+\mu_{i,j}r_jm_j +m/q-m/2 \le\\ \le\max \{(1-\lambda_{i,j})r_im_i+\lambda_{i,j}r_jm_j , \, r_jm_j+m/q-m/p_j \} \; \text{if }p_i> q,
\end{array}
\end{align}
\begin{align}
\label{1tlijriti1}
\begin{array}{c}
 (1-\mu_{i,j})r_im_i+\mu_{i,j}r_jm_j +m/q-m/2 \le \\ \le\max \{ r_im_i+m/q-m/p_i, \,  r_jm_j+m/q-m/p_j\} \; \text{if } 2< p_i\le q.
\end{array}
\end{align}

Hence, both for $q\le 2$ and for $q>2$, 
\begin{align}
\label{del_m_d0} C_{\overline{m}} \underset{\overline{p},q,\overline{r},d}{\lesssim} 2^{-\varphi(m_1, \, \dots, \, m_d)},
\end{align}
where
\begin{align}
\label{phi_m_def} \begin{array}{c} \varphi(t_1, \, \dots, \, t_d) = \max \Bigl\{\max _{p_j\ge q} t_jr_j, \, \max _{p_j\le q} (t_jr_j+t/q-t/p_j), \\ \max_{p_i> q, \, p_j< q} ((1-\lambda_{i,j})r_it_i+\lambda_{i,j}r_jt_j)\Bigr\}, \quad
t=t_1+\dots+t_d.
\end{array}
\end{align}

We define the function $h$ and the set $D$ by formulas (\ref{h_def_00}) and (\ref{d_set_def_00}) (both for $q\le 2$ and for $q>2$).

\begin{Lem}
\label{rem_est_h} Let $(\alpha_1^*, \, \dots, \, \alpha_d^*)$ be the minimum point of the function $h$ on $D$, and let $h(\alpha_1^*, \, \dots, \, \alpha_d^*)>0$. The numbers $C_{\overline{m}}$ are defined by \eqref{cmdef}. Then for all $N\in \N$
\begin{align}
\label{rem_est_gen} \sum \limits _{m\ge N} \|\delta_{\overline{m}}f\|_{L_q(\mathbb{T}^d)} \underset{\overline{p},q,\overline{r},d}{\lesssim} \sum \limits _{m\ge N} C_{\overline{m}} \underset{\overline{p},q,\overline{r},d}{\lesssim} 2^{-N\cdot h(\alpha_1^*, \, \dots, \, \alpha_d^*)} N^{d-1}.
\end{align}
If the minimum point of the function $h$ on $D$ is unique, then 
\begin{align}
\label{rem_est_gen1} \sum \limits _{m\ge N} \|\delta_{\overline{m}}f\|_{L_q(\mathbb{T}^d)} \underset{\overline{p},q,\overline{r},d}{\lesssim} \sum \limits _{m\ge N} C_{\overline{m}} \underset{\overline{p},q,\overline{r},d}{\lesssim} 2^{-N\cdot h(\alpha_1^*, \, \dots, \, \alpha_d^*)}.
\end{align}
\end{Lem}
\begin{proof}
The first order inequality in \eqref{rem_est_gen}, \eqref{rem_est_gen1} follows from \eqref{cmdef}.

Let us prove the second inequality in \eqref{rem_est_gen}, \eqref{rem_est_gen1}. From (\ref{del_m_d0}) we get:
$$
\sum \limits _{m\ge N} C_{\overline{m}} \underset{\overline{p},q,\overline{r},d}{\lesssim} \sum \limits _{m\ge N} 2^{-\varphi(\overline{m})} \underset{\overline{p},q,\overline{r},d}{\lesssim} \int \limits _{t\ge N, \, t_1, \dots,t_d\ge 0} 2^{-\varphi(t_1, \, \dots, \, t_d)}\, dt_1\dots dt_d =:\Sigma,
$$
where $t=t_1+\dots+t_d$. We set $\alpha_j=\frac{t_j}{t}$, $1\le j\le d$. Comparing \eqref{phi_m_def} and the definition of the function $h$, we see that $\varphi(t_1, \, \dots, \, t_d) = t \cdot h(\alpha_1, \, \dots, \, \alpha_d)$, $(\alpha_1, \, \dots, \, \alpha_d)\in D$. Let $E_t =\{(t_1, \, \dots, \, t_{d-1}):\; t_1+\dots+t_{d-1}\le t, \; t_j\ge 0, \; 1\le j\le d-1\}$. We have
$$
\Sigma = \int \limits _N^\infty \int \limits_{E_t}2^{-\varphi(t_1, \, \dots, \, t_{d-1}, \, t-t_1-\dots-t_{d-1})}\, dt_1\dots dt_{d-1}\, dt \le 
$$
$$
\le \int \limits _N^\infty 2^{-t\cdot h(\alpha_1^*, \, \dots, \, \alpha_d^*)} t^{d-1}\, dt \underset{\overline{p},q,\overline{r},d}{\lesssim} 2^{-N\cdot h(\alpha_1^*, \, \dots, \, \alpha_d^*)}N^{d-1}.
$$

If the minimum point of the function $h$ on $D$ is unique, then $(N\alpha_1^*, \, \dots, \, N\alpha_d^*)$ is the unique minimum point of the function $\varphi$ on $G_N:=\{(t_1, \, \dots, \, t_d):\; t_1+\dots+t_d\ge N, \; t_j\ge 0, \, 1\le j\le d\}$. From \eqref{phi_m_def} it follows that there is a number $b=b(\overline{p}, \, q, \, \overline{r}, \, d)> 0$ such that $\varphi(t_1, \, \dots, \, t_d)\ge \varphi(N\alpha_1^*, \, \dots, \, N\alpha_d^*) + b\sum \limits _{j=1}^d |t_j-N\alpha_j^*|$, $(t_1, \, \dots, \, t_d)\in G_N$. Therefore,
 $$\Sigma \underset{\overline{p},q,\overline{r},d}{\lesssim} 2^{-\varphi(N\alpha_1^*, \, \dots, \, N\alpha_d^*)} = 2^{-N\cdot h(\alpha_1^*, \, \dots, \, \alpha_d^*)}.$$
This completes the proof.
\end{proof}

Let $q>2$, and let the function $\tilde h$ and the set $\tilde D$ be defined in part 2 of Theorem \ref{main}. We denote
$$
\tilde D_{q/2} = \{(\alpha_1, \, \dots, \, \alpha_d, \, s)\in \tilde D:\; s=q/2\}.
$$

\begin{Lem}
\label{emb_qg2} Let $q>2$. Then for each $(\alpha_1, \, \dots, \, \alpha_d, \, q/2)\in \tilde D_{q/2}$ we have 
\begin{align}
\label{htilh} \tilde h(\alpha_1, \, \dots, \, \alpha_d, \, q/2) = \frac q2 \cdot h(2\alpha_1/q, \, \dots, \, 2\alpha_d/q).
\end{align}
In particular,
\begin{align}
\label{htilh1}
\min _{\tilde D_{q/2}}\tilde h = \frac q2 \min _D h.
\end{align}
\end{Lem}

\begin{proof}
Let us check (\ref{htilh}). We have
$$
\tilde h(\alpha_1, \, \dots, \, \alpha_d, \, q/2) = \max \Bigl\{ \max _{p_j\ge q} r_j\alpha_j, \, \max _{p_j\le q} (r_j\alpha_j +1/2 - q/2p_j),
$$
$$
\max _{p_i> q, \, p_j< q} ((1-\lambda_{i,j})r_i\alpha_i + \lambda_{i,j}r_j\alpha_j), \, \max _{p_i> 2, \, p_j< 2} ((1-\mu_{i,j})r_i\alpha_i + \mu_{i,j}r_j\alpha_j+1/2-q/4)\Bigr\}=
$$
$$
= \max \Bigl\{ \max _{p_j\ge q} r_j\alpha_j, \, \max _{p_j\le q} (r_j\alpha_j +1/2 - q/2p_j),\, \max _{p_i> q, \, p_j< q} ((1-\lambda_{i,j})r_i\alpha_i + \lambda_{i,j}r_j\alpha_j)\Bigr\}=
$$
$$
=\frac q2 \cdot h(2\alpha_1/q, \, \dots, \, 2\alpha_d/q);
$$
the second equality follows from \eqref{1tlijriti}, \eqref{1tlijriti1}. This completes the proof.
\end{proof}

\begin{Lem}
\label{up_est_lem}
Let $\min _D h>0$, and let for each
$\overline{m}\in \N^d$ the number $k_{\overline{m}}\in \Z_+$ be given. Suppose that there is a number $C\in \N$ such that $\sum \limits _{\overline{m}\in \N^d}k_{\overline{m}}\le Cn$. Then
\begin{align}
\label{per_up} d_{Cn}(W^{\overline{r}}_{\overline{p}}(\mathbb{T}^d), \, L_q(\mathbb{T}^d)) \underset{\overline{p},q,\overline{r},d}{\lesssim} \sum \limits _{\overline{m}\in \N^d} d_{k_{\overline{m}}}\left(\cap _{j=1}^d 2^{-m_jr_j-m/q+m/p_j} B^{2^m}_{p_j}, \, l_q^{2^m}\right).
\end{align}
\end{Lem}
\begin{proof}
Since $\min _{D}h>0$, it follows from
\eqref{rem_est_gen} that the sequence of partial sums $S_Nf:=\sum \limits _{m\le N} \delta _{\overline{m}}f$ is fundamental in $L_q(\mathbb{T}^d)$. On the other hand, $S_Nf\underset{N\to \infty}{\to} f$ in ${\cal S}'(\mathbb{T}^d)$. Hence $f\in L_q(\mathbb{T}^d)$ and $S_Nf\underset{N\to \infty}{\to} f$ in $L_q(\mathbb{T}^d)$; this yields
$$
f = \sum \limits _{N\in \N} \sum \limits _{\overline{m}\in \N^d:\, m=N} \delta _{\overline{m}}f
$$
(the series converges in $L_q(\mathbb{T}^d)$). It remains to apply Lemma \ref{up_est_lem_m}.
\end{proof}

\renewcommand{\proofname}{\bf Proof of Theorem \ref{main}}
\begin{proof} First we prove the upper estimate under assumption that $\min _{D}h>0$. Then we prove the lower estimate and obtain that if $\min _{D}h\le 0$, then $W^{\overline{r}}_{\overline{p}}(\mathbb{T}^d)$ is not compactly embedded  into $L_q(\mathbb{T}^d)$. Hence, under the conditions of Theorem \ref{main}, the inequality $\min _{D}h>0$ holds, since the inequality $\frac{\langle \overline{r}\rangle}{d}+ \frac 1q - \frac{\langle \overline{r}\rangle}{\langle \overline{p}\circ\overline{r}\rangle}>0$ implies the compact embedding.

Let $\min _{D}h>0$.

We set $q_*=\min\{q, \, 2\}$. Let $\overline{m}_*=(m_1^*, \, \dots, \, m_d^*)\in \R_+^d$, $2^{m_*}\in [n, \, n^{q_*/2}]$, $\varepsilon>0$ ($\overline{m}_*$ and $\varepsilon$ will be defined later from $\overline{p}$, $q$, $\overline{r}$ and $d$). We denote $$|\overline{m}-\overline{m}_*| := \sum \limits _{j=1}^d |m_j-m_j^*|.$$ Let 
\begin{align}
\label{k_vec_m_def}
k_{\overline{m}}= \begin{cases} 0 &\text{ for }2^m > n^{q_*/2} \\ \min \{\lfloor n\cdot 2^{-\varepsilon|\overline{m}-\overline{m}_*|}\rfloor, \, 2^m\} & \text{ for } 2^m \le n^{q_*/2}.\end{cases}
\end{align}
 Then
$$
\sum \limits _{\overline{m}\in \N^d} k_{\overline{m}} \le \sum \limits _{\overline{m}\in \N^d} n\cdot 2^{-\varepsilon|\overline{m}-\overline{m}_*|} \underset{\varepsilon, \, d}{\lesssim} n.
$$

We apply Lemma \ref{up_est_lem} and get the upper estimate for the right-hand side of \eqref{per_up}:
$$
\sum \limits _{\overline{m}\in \N^d} d_{k_{\overline{m}}}\left(\cap _{j=1}^d 2^{-m_jr_j-m/q+m/p_j} B^{2^m}_{p_j}, \, l_q^{2^m}\right) \le
$$
$$
\le \sum \limits _{2^m > n^{q_*/2}} d_0\left(\cap _{j=1}^d 2^{-m_jr_j-m/q+m/p_j} B^{2^m}_{p_j}, \, l_q^{2^m}\right)+
$$
$$
+\sum \limits _{ n\cdot 2^{-\varepsilon|\overline{m}-\overline{m}_*|}\le 2^m \le n^{q_*/2}} d_{k_{\overline{m}}}\left(\cap _{j=1}^d 2^{-m_jr_j-m/q+m/p_j} B^{2^m}_{p_j}, \, l_q^{2^m}\right).
$$

We consider the case $q>2$ (the case $q\le 2$ is simpler and can be considered similarly). Let $$S_1 = \sum \limits _{2^m > n^{q/2}} d_0\left(\cap _{j=1}^d 2^{-m_jr_j-m/q+m/p_j} B^{2^m}_{p_j}, \, l_q^{2^m}\right),$$
$$
S_{2,\varepsilon} = \sum \limits _{n\cdot 2^{-\varepsilon|\overline{m}-\overline{m}_*|}\le 2^m \le n^{q/2}} d_{k_{\overline{m}}}\left(\cap _{j=1}^d 2^{-m_jr_j-m/q+m/p_j} B^{2^m}_{p_j}, \, l_q^{2^m}\right).
$$
We prove that
\begin{align}
\label{s1est0} S_1 \underset{\overline{p},q,\overline{r},d}{\lesssim} n^{-\tilde h(\hat \alpha_1, \, \dots, \, \hat \alpha_d, \, \hat s)}.
\end{align}
We denote $c_*=\min _{\tilde D_{q/2}}\tilde h$, $\log x := \log _2 x$.
Applying \eqref{cmdef}, \eqref{rem_est_gen} with $N=\frac q2 \log n$, we get
$$
S_1 \underset{\overline{p},q,\overline{r},d}{\lesssim} 2^{-\frac{q \log n}{2}\min _D h}(\log n)^{d-1} \stackrel{\eqref{htilh1}}{=} n^{- c_*}(\log n)^{d-1}.$$
If $c_*> \tilde h(\hat \alpha_1, \, \dots, \, \hat\alpha_d, \, \hat s)$, then $n^{- c_*}(\log n)^{d-1} \underset{\overline{p},q,\overline{r},d}{\lesssim} n^{-\tilde h(\hat \alpha_1, \, \dots, \, \hat \alpha_d, \, \hat s)}$; this implies \eqref{s1est0}. If $c_*= \tilde h(\hat \alpha_1, \, \dots, \, \hat \alpha_d, \, \hat s)$, then the function $\tilde h$ has a unique minimum point on $\tilde D_{q/2}$; by Lemma \ref{emb_qg2}, the function $h$ also has a unique minimum point on $D$. Applying \eqref{rem_est_gen1}, we again get \eqref{s1est0}.

Now we estimate $S_{2,\varepsilon}$. First we consider the case $\varepsilon = 0$:
$$
S_{2,0}=\sum \limits _{n\le 2^m\le n^{q/2}} d_n\left(\cap _{j=1}^d 2^{-m_jr_j-m/q+m/p_j} B^{2^m}_{p_j}, \, l_q^{2^m}\right).
$$
Applying Theorems \ref{glus}, \ref{p_s} and \ref{fin_dim_inters}, we get
$$
S_{2,0}\underset{q}{\lesssim} \sum \limits _{n\le 2^m\le n^{q/2}} 2^{-\psi_n(\overline{m}, \, m)},
$$
where
\begin{align}
\label{psi_n_f_def_00}
\begin{array}{c}
\psi_n(t_1, \, \dots, \, t_d, \, t) = \max \Bigl\{ \max _{j\in I} r_jt_j, \, \max _{j\in J} \Bigl( r_jt_j -\frac 12\cdot \frac{1/p_j-1/q}{1/2-1/q}t + \frac 12\cdot \frac{1/p_j-1/q}{1/2-1/q} \log n\Bigr),
\\
\max _{j\in K} \Bigl(r_jt_j-\frac{t}{p_j} + \frac 12 \log n\Bigr), \, \max _{i\in I', \, j\in J'\cup K}((1-\lambda_{i,j})r_it_i + \lambda_{i,j}r_jt_j),
\\
\max _{i\in I\cup J', \, j\in K'} \Bigl((1-\mu_{i,j})r_it_i+\mu_{i,j}r_jt_j -\frac t2 + \frac 12\log n\Bigr)\Bigr\}.
\end{array}
\end{align}
We set
$$
G_n =\Bigl\{(t_1, \, \dots, \, t_d)\in \R^d_+:\; \log n \le t_1+\dots+t_d \le \frac q2 \log n\Bigr\}.
$$
Then
$$
\sum \limits _{n\le 2^m\le n^{q/2}} 2^{-\psi_n(\overline{m}, \, m)} \underset{\overline{p}, \, \overline{r}, \, q, \, d}{\lesssim} \int \limits _{G_n} 2^{-\psi_n(t_1, \, \dots, \, t_d, \, t_1+\dots+t_d)}\, dt_1\dots dt_d.
$$
Notice that 
\begin{align}
\label{psi_n_t1td}
\psi_n(t_1, \, \dots, \, t_d, \, t) = \tilde h (t_1/\log n, \, \dots, \, t_d/\log n, \, t/\log n)\cdot\log n. 
\end{align}
Since the minimum point of the function $\tilde h$ on the set $\tilde D$ is unique, the minimum point of the function $f_n(t_1, \, \dots, \, t_d):=\psi_n(t_1, \, \dots, \, t_d, \, t_1+\dots+t_d)$ on $G_n$ is also unique and has the form $(\hat \alpha_1, \, \dots, \, \hat \alpha_d)\log n$. In addition, if $\hat s=\hat \alpha_1+\dots + \hat \alpha_d>1$, then $(\hat \alpha_1, \, \dots, \, \hat \alpha_d)\log n$ in the unique minimum point of the function $f_n$ on the set
$$
\hat G_n =\Bigl\{(t_1, \, \dots, \, t_d)\in \R^d_+:\; t_1+\dots+t_d \le \frac q2 \log n\Bigr\}.
$$

We denote 
\begin{align}
\label{m_st_m1a1ln}
\overline{m}_*=(m_1^*, \, \dots, \, m_d^*) = (\hat \alpha_1, \, \dots, \, \hat \alpha_d)\log n.
\end{align}

By (\ref{psi_n_f_def_00}), (\ref{psi_n_t1td}), there is $c_{\overline{p},q,\overline{r},d} > 0$ such that
\begin{align}
\label{psi_t_low} \begin{array}{c}\psi_n(t_1, \, \dots, \, t_d, \, t_1+\dots+t_d) \ge \tilde h(\hat \alpha_1, \, \dots, \, \hat \alpha_d, \, \hat s)\log n + c_{\overline{p},q,\overline{r},d} \sum \limits _{j=1}^d |t_j-m^*_j|, \\ (t_1, \, \dots, \, t_d)\in \begin{cases} G_n, & \hat s = 1, \\ \hat G_n, & \hat s > 1.\end{cases}\end{array}
\end{align}

This implies the estimate
$$
\int \limits _{G_n} 2^{-\psi_n(t_1, \, \dots, \, t_d, \, t_1+\dots+t_d)}\, dt_1\dots dt_d \underset{\overline{p}, \, \overline{r}, \, q, \, d}{\lesssim} n^{-\tilde h(\hat \alpha_1, \, \dots, \, \hat \alpha_d, \, \hat s)};
$$
i.e., $S_{2,0}\underset{\overline{p}, \, \overline{r}, \, q, \, d}{\lesssim} n^{-\tilde h(\hat \alpha_1, \, \dots, \, \hat \alpha_d, \, \hat s)}$.

Now we get the estimates for $S_{2,\varepsilon}$, where $\varepsilon >0$ is sufficiently small. We set $$G_{n,\varepsilon} = \Bigl\{(t_1, \, \dots, \, t_d)\in \R^d_+:\; \log n - \varepsilon \sum \limits _{j=1}^d |t_j-m_j^*| \le t_1+\dots+t_d\le \frac q2 \log n\Bigr\}.$$
From \eqref{psi_t_low} it follows that for small $\varepsilon>0$
\begin{align}
\label{psi_estlow2} \begin{array}{c}\psi_n(t_1, \, \dots, \, t_d, \, t_1+\dots+t_d) \ge \tilde h(\hat \alpha_1, \, \dots, \, \hat \alpha_d, \, \hat s)\log n + \frac{c_{\overline{p},q,\overline{r},d}}{2} \sum \limits _{j=1}^d |t_j-m^*_j|, \\ (t_1, \, \dots, \, t_d) \in G_{n,\varepsilon}\end{array}
\end{align}
(the maximal possible value of $\varepsilon$, for which \eqref{psi_estlow2} holds, is determined from $\overline{p}$, $q$, $\overline{r}$, $d$).

Now we apply Theorems \ref{glus}, \ref{p_s} and \ref{fin_dim_inters} together with \eqref{k_vec_m_def}. Since $2^{m_*}\in [n, \, n^{q/2}]$, we have $k_n \asymp n\cdot 2^{-\varepsilon|\overline{m}-\overline{m}_*|}$ for sufficiently small $\varepsilon>0$ for $n\cdot 2^{-\varepsilon|\overline{m}-\overline{m}_*|}\le 2^m \le n^{q/2}$. This yields that for some $b=b(\overline{p}, \, \overline{r}, \, q, \, d)>0$
$$
S_{2,\varepsilon} \underset{q}{\lesssim} \sum \limits _{n\cdot 2^{-\varepsilon|\overline{m}-\overline{m}_*|}\le 2^m \le n^{q/2}} 2^{-\psi_n(\overline{m}, \, m)}\cdot 2^{\varepsilon b |\overline{m}-\overline{m}_*|} \underset{\overline{p}, \, \overline{r}, \, q, \, d}{\lesssim} 
$$
$$
\lesssim \int \limits _{G_{n,\varepsilon}} 2^{-\psi(t_1, \, \dots, \, t_d, \, t_1+\dots+t_d) + \varepsilon b\sum \limits _{j=1}^d |t_j-m^*_j|}\, dt_1\dots dt_d \stackrel{\eqref{psi_estlow2}}{\le}
$$
$$
\le \int \limits _{G_{n,\varepsilon}} 2^{-\tilde h(\hat \alpha_1, \, \dots, \, \hat \alpha_d, \, \hat s)\log n - (c_{\overline{p},q,\overline{r},d}/2-b\varepsilon) \sum \limits _{j=1}^d |t_j-m^*_j|}\, dt_1\dots dt_d \underset{\overline{p}, \, \overline{r}, \, q, \, d}{\lesssim}
$$
$$
\lesssim n^{-\tilde h(\hat \alpha_1, \, \dots, \, \hat \alpha_d, \, \hat s)},
$$
if $\varepsilon >0$ is small. This together with \eqref{s1est0} yields the desired estimate for the widths.

Now we obtain the lower estimate. We again consider the more complicated case $q>2$. Let $\overline{m} \in \N^d$, $2n\le 2^m \le n^{q/2}$, $\alpha_j = m_j/\log n$, $1\le j\le d$, $s=\alpha_1+\dots+\alpha_d$. From Lemma \ref{low_est_lem} and Theorems \ref{glus}, \ref{p_s}, \ref{fin_dim_inters} we get
\begin{align}
\label{dnwlow_est}
\begin{array}{c}
d_n(W^{\overline{r}}_{\overline{p}}(\mathbb{T}^d), \, L_q(\mathbb{T}^d)) \underset{\overline{p}, \, \overline{r}, \, q, \, d}{\gtrsim} d_n \Bigl(\cap _{i=1}^d 2^{-m_ir_i-m/q+m/p_i}B^{2^m}_{p_i}, \, l_q^{2^m}\Bigr) \underset{q}{\asymp} \\ \asymp 2^{-\psi_n(\overline{m}, \, m)} \stackrel{\eqref{psi_n_t1td}}{=} n^{-\tilde h( \alpha_1, \, \dots, \, \alpha_d, \, s)}.
\end{array}
\end{align}
Let the point $\overline{m}_*\in G_n$ be defined by \eqref{m_st_m1a1ln}, and let $\overline{m}\in G_n$ be the nearest point to $\overline{m}_*$ (with respect to Euclidean norm) with integer positive coordinates, such that $m\ge \log(2n)$. Then
\begin{align}
\label{nth_nth}
n^{-\tilde h( \alpha_1, \, \dots, \, \alpha_d, \, s)} \underset{\overline{p}, \, \overline{r}, \, q, \, d}{\asymp} n^{\tilde h(\hat \alpha_1, \, \dots, \, \hat\alpha_d, \, \hat s)},
\end{align}
which implies the desired lower estimate for the widths.

If $\min _D h\le 0$, then $\min _{\tilde D} \tilde h\le\min _{\tilde D_{q/2}} \tilde h\le 0$ (see Lemma \ref{emb_qg2}); this together with \eqref{dnwlow_est}, \eqref{nth_nth} yields that $$d_n(W^{\overline{r}}_{\overline{p}}(\mathbb{T}^d), \, L_q(\mathbb{T}^d)) \underset{\overline{p}, \, \overline{r}, \, q, \, d}{\gtrsim} 1;$$ i.e., $W^{\overline{r}}_{\overline{p}}(\mathbb{T}^d)$ is not compactly embedded into $L_q(\mathbb{T}^d)$.

This completes the proof.
\end{proof}

\begin{Rem}
\label{h_otr} It was proved that if $\min _D h\le 0$, then $d_n(W^{\overline{r}}_{\overline{p}}(\mathbb{T}^d), \, L_q(\mathbb{T}^d)) \underset{\overline{p}, \, \overline{r}, \, q, \, d}{\gtrsim} 1$. For $q>2$, the same assertion holds if $\min _{\tilde D} \tilde h\le 0$; it follows from \eqref{dnwlow_est}, \eqref{nth_nth}.
\end{Rem}

\renewcommand{\proofname}{\bf Proof}

\section{Proof of Theorem \ref{main1}}

By Theorem \ref{main}, it suffices to find the minimum point of the function $h$ on $D$ for $q\le 2$, and the minimum point of the function $\tilde h$ on $\tilde D$ for $q>2$.

We consider the more complicated case $q>2$ (for $q\le 2$, the arguments are similar; here we use Lemma \ref{expl1}).

First we prove Theorem \ref{main1} under the following assumption: $p_i\notin \{2, \, q\}$, $1\le i\le d$. Then $I=I'$, $J=J'$, $K=K'$.

In what follows, the numbers $\lambda_{i,j}$ and $\mu_{i,j}$ are defined by \eqref{lam_ij_mu_ij_def}.

\begin{Lem}
\label{oblast}
Let $p_i\notin \{2, \, q\}$, $\alpha_i\ge 0$ $(1\le i\le d)$, $\alpha_1+ \dots+\alpha_d=s$, $1\le s\le q/2$.
\begin{enumerate}
\item Let $j\in I$. Then $\tilde h(\alpha_1, \, \dots, \, \alpha_d, \, s)= r_j\alpha_j$ if and only if $\alpha_jr_j - \alpha_ir_i\ge 0$, $1\le i\le d$.
\item Let $j\in J$. Then $$\tilde h(\alpha_1, \, \dots, \, \alpha_d, \, s)=r_j\alpha_j - \frac 12\cdot \frac{1/p_j-1/q}{1/2-1/q}(s-1)$$
if and only if
$$
\alpha_jr_j - \alpha_ir_i\ge \frac 12\cdot \frac{1/p_j-1/p_i}{1/2-1/q}(s-1), \quad 1\le i\le d.
$$
\item Let $j\in K$. Then $$\tilde h(\alpha_1, \, \dots, \, \alpha_d, \, s)=r_j\alpha_j - \frac{s}{p_j} + \frac 12$$
if and only if
$$
\alpha_jr_j - \alpha_ir_i\ge \frac{s}{p_j} - \frac{s}{p_i}, \quad 1\le i\le d.
$$
\item Let $i\in I$, $j\in J\cup K$. Then
$$
\tilde h(\alpha_1, \, \dots, \, \alpha_d, \, s)=(1-\lambda_{i,j})r_i\alpha_i+ \lambda_{i,j} r_j\alpha_j
$$
if and only if
$$
\alpha_ir_i - \alpha_jr_j\le 0, \quad \alpha_ir_i - \alpha_jr_j\ge \frac 12\cdot \frac{1/p_i-1/p_j}{1/2-1/q}(s-1),
$$
$$
\frac{\alpha_ir_i - \alpha_jr_j}{1/p_i-1/p_j} \ge \frac{\alpha_ir_i - \alpha_kr_k}{1/p_i-1/p_k}, \quad k\in J\cup K,
$$
$$
\frac{\alpha_ir_i - \alpha_jr_j}{1/p_i-1/p_j} \le \frac{\alpha_kr_k - \alpha_jr_j}{1/p_k-1/p_j}, \quad k\in I.
$$
\item Let $i\in I\cup J$, $j\in K$. Then
$$
\tilde h(\alpha_1, \, \dots, \, \alpha_d, \, s)=(1-\mu_{i,j})r_i\alpha_i+ \mu_{i,j} r_j\alpha_j -\frac s2 + \frac 12
$$
if and only if
$$
\alpha_ir_i - \alpha_jr_j\le \frac 12\cdot \frac{1/p_i-1/p_j}{1/2-1/q}(s-1), \quad \alpha_ir_i - \alpha_jr_j\ge \frac{s}{p_i} - \frac{s}{p_j},
$$
$$
\frac{\alpha_ir_i - \alpha_jr_j}{1/p_i-1/p_j} \ge \frac{\alpha_ir_i - \alpha_kr_k}{1/p_i-1/p_k}, \quad k\in K,
$$
$$
\frac{\alpha_ir_i - \alpha_jr_j}{1/p_i-1/p_j} \le \frac{\alpha_kr_k - \alpha_jr_j}{1/p_k-1/p_j}, \quad k\in I\cup J.
$$
\end{enumerate}
\end{Lem}
\begin{proof}
First we prove that the corresponding conditions are necessary. In case 1, we use the inequalities $\alpha_jr_j\ge \alpha_ir_i$ for $i\in I$ and $r_j\alpha_j \ge (1-\lambda_{j,i})r_j\alpha_j+\lambda_{j,i}r_i\alpha_i$ for $i\in J\cup K$. In case 2, we apply the inequalities
$$
r_j\alpha_j - \frac 12\cdot \frac{1/p_j-1/q}{1/2-1/q}(s-1) \ge r_i\alpha_i - \frac 12\cdot \frac{1/p_i-1/q}{1/2-1/q}(s-1), \quad i\in J,
$$
$$
r_j\alpha_j - \frac 12\cdot \frac{1/p_j-1/q}{1/2-1/q}(s-1) \ge (1-\lambda_{i,j})r_i\alpha_i+\lambda_{i,j}r_j\alpha_j, \quad i\in I,
$$
$$
r_j\alpha_j - \frac 12\cdot \frac{1/p_j-1/q}{1/2-1/q}(s-1) \ge (1-\mu_{j,i})r_j\alpha_j + \mu_{j,i}r_i\alpha_i -\frac s2 + \frac 12, \quad i\in K.
$$
In case 3, we use the inequalities
$$
r_j\alpha_j+\frac 12 -\frac{s}{p_j} \ge r_i\alpha_i+\frac 12 -\frac{s}{p_i}, \quad i\in K,
$$
$$
r_j\alpha_j+\frac 12 -\frac{s}{p_j} \ge (1-\mu_{i,j})\alpha_ir_i + \mu_{i,j} \alpha_jr_j +\frac 12 -\frac s2, \quad i\in I\cup J.
$$
In case 4, we use the inequalities
$$
(1-\lambda_{i,j})r_i\alpha_i+ \lambda_{i,j} r_j\alpha_j\ge r_i\alpha_i,
$$
$$
(1-\lambda_{i,j})r_i\alpha_i+ \lambda_{i,j} r_j\alpha_j\ge \alpha_jr_j - \frac 12\cdot \frac{1/p_j-1/q}{1/2-1/q}(s-1) \quad \text{for }j\in J,
$$
$$
(1-\lambda_{i,j})r_i\alpha_i+ \lambda_{i,j} r_j\alpha_j\ge (1-\mu_{i,j})r_i\alpha_i+ \mu_{i,j} r_j\alpha_j - \frac 12(s-1)\quad \text{for }j\in K,
$$
$$
(1-\lambda_{i,j})r_i\alpha_i+ \lambda_{i,j} r_j\alpha_j\ge (1-\lambda_{i,k})r_i\alpha_i+ \lambda_{i,k} r_k\alpha_k, \quad k\in J\cup K,
$$
$$
(1-\lambda_{i,j})r_i\alpha_i+ \lambda_{i,j} r_j\alpha_j\ge (1-\lambda_{k,j})r_k\alpha_k+ \lambda_{k,j} r_j\alpha_j, \quad k\in I.
$$
Case 5 is similar to case 4.

Now we prove that the conditions are sufficient. We consider case 1 (the other cases are similar). Let $\alpha_1^*+\dots+\alpha_d^*=s^*\in [1, \, q/2]$, $\alpha_j^*\ge 0$ ($1\le j\le d$), and let $r_j\alpha_j^*\ge r_i\alpha_i^*$ for all $i=1, \, \dots, \, d$, but $\tilde h(\alpha_1^*, \, \dots, \, \alpha_d^*, \, s^*) > r_j\alpha_j^*$. Then there exist $c>0$ and an open subset $U$ of the set $$\{(\alpha_1, \, \dots, \, \alpha_d, \, s):\alpha_1+\dots+\alpha_d=s,\; \alpha_i\ge 0, \; 1\le s\le q/2,\; r_j\alpha_j\ge r_i\alpha_i, \, 1\le i\le d\}$$ such that for all $(\alpha_1, \, \dots, \, \alpha_d, \, s)\in U$ the inequality
\begin{align}
\label{til_h_rjaj_c}
\tilde h(\alpha_1, \, \dots, \, \alpha_d, \, s)-r_j\alpha_j \ge c 
\end{align}
holds.

For sufficiently large $n\in \N$ there is $(\alpha_1, \, \dots, \, \alpha_d, \, s)\in U$ such that $m_k:=\alpha_k\log n \in \N$ ($1\le k\le d$), $m\ge \log(2n)$; recall that $m:=m_1+\dots+m_d$. Then
$$
2^{-r_jm_j-m/q+m/p_j} \cdot 2^{m(1/p_i-1/p_j)} \le 2^{-r_im_i-m/q+m/p_i}, \quad 1\le i\le d.
$$
By Lemma \ref{expl2} (case 2),
$$
d_n(\cap _{i=1}^d 2^{-r_im_i-m/q+m/p_i}B_{p_i}^{2^m}, \, l_q^{2^m}) \underset{q}{\asymp} 2^{-m_jr_j} = n^{-r_j\alpha_j}.
$$
On the other hand, by \eqref{dnwlow_est}, $$d_n(\cap _{i=1}^d 2^{-r_im_i-m/q+m/p_i}B_{p_i}^{2^m}, \, l_q^{2^m}) \underset{q}{\asymp} n^{-\tilde h(\alpha_1, \, \dots, \, \alpha_d, \, s)}\stackrel{\eqref{til_h_rjaj_c}}{\le} n^{-r_j\alpha_j-c}.$$ We arrived to a contradiction.
\end{proof}

{\bf Proof of Theorem \ref{main1} in the case $p_i\notin\{2, \, q\}$, $1\le i\le d$.}

We define the points
\begin{align}
\label{xi_k_def14}
\xi_k = (\alpha_1^k, \, \dots, \, \alpha_d^k, \, s^k), \quad 1\le k\le 4,
\end{align}
as follows: $s^1=s^2=1$, $s^3=s^4=\frac q2$,
\begin{align}
\label{aj12}
\alpha^1_j = \frac{1/r_j}{\sum _{i=1}^d 1/r_i}, \quad
\alpha^2_j = \frac{1-\sum \limits _{i=1}^d\frac{1}{r_i}(1/p_i-1/p_j)}{r_j\sum \limits _{i=1}^d 1/r_i}, \quad 1\le j\le d,
\end{align}
\begin{align}
\label{aj34}
\alpha^3_j=\frac{q}{2}\alpha_j^2, \quad \alpha^4_j = \frac{q}{2}\alpha_1^j, \quad 1\le j\le d.
\end{align}
Notice that by (\ref{dop_usl}) we have $\alpha^k_j>0$, $1\le k\le 4$, $1\le j\le d$.

We show that the minimum of the function $\tilde h$ on $\tilde D$ is attained only in $\xi_1$, $\xi_2$ or $\xi_3$, and evaluate $\tilde h$ in these points.

We introduce some more notation. Let $\hat l_{m,t}$ be the segments that join $\xi_m$ and $\xi_t$, $1\le m, \, t\le 4$, $m\ne t$. For $1\le k\le d$ we define the segments $l_k$, $\tilde l_k$ and $\hat l_k$ as follows: $l_k$ are defined by the conditions
\begin{align}
\label{l_k_def} \begin{array}{c} r_1\alpha_1=\dots=r_{k-1}\alpha_{k-1}= r_{k+1}\alpha_{k+1}=\dots =r_d\alpha_d,\\ \alpha_1+\dots +\alpha_d=s=1, \; r_k\alpha_k- r_j\alpha_j \le 0,\;j\ne k,\; \alpha_i\ge 0, \; 1\le i\le d, \end{array}
\end{align}
$\hat l_k$ are defined by
\begin{align}
\label{hat_l_k_def} 
\begin{array}{c}
r_1\alpha_1-\frac{1}{p_1} = \dots = r_{k-1}\alpha_{k-1}-\frac{1}{p_{k-1}} = r_{k+1}\alpha_{k+1}-\frac{1}{p_{k+1}}=\dots = r_d\alpha_d-\frac{1}{p_d},
\\
\alpha_1+\dots+\alpha_d=s=1,
\\
r_k\alpha_k-r_j\alpha_j\le \frac{1}{p_k}-\frac{1}{p_j}, \;j\ne k,\; \alpha_i\ge 0, \; 1\le i\le d,
\end{array}
\end{align}
and $\tilde l_k$, by
\begin{align}
\label{til_l_k_def} 
\begin{array}{c}
r_1\alpha_1-\frac{q}{2p_1} = \dots = r_{k-1}\alpha_{k-1}-\frac{q}{2p_{k-1}} = r_{k+1}\alpha_{k+1}-\frac{q}{2p_{k+1}}=\dots = r_d\alpha_d-\frac{q}{2p_d},
\\
\alpha_1+\dots+\alpha_d=s=q/2, \\
r_k\alpha_k-r_j\alpha_j\le \frac{q}{2p_k}-\frac{q}{2p_j}, \;j\ne k,\; \alpha_i\ge 0, \; 1\le i\le d.
\end{array}
\end{align}

Notice that $\xi_1\in l_k$, $\xi_2\in \hat l_k$, $\xi_3\in \tilde l_k$ (they are endpoints of the corresponding segments); the systems of equations and inequalities \eqref{l_k_def}--\eqref{til_l_k_def} have the same matrices. Hence the segments $l_k$, $\tilde l_k$ and $\hat l_k$ have the form
\begin{align}
\label{lk_vk} l_k = \xi_1 + tv_k, \; 0\le t\le \tau_k; \; \hat l_k = \xi_2 + tv_k, \; 0\le t\le \hat\tau_k; \; \tilde l_k = \xi_3 + tv_k, \; 0\le t\le \tilde \tau_k;
\end{align}
here $\tau_k$, $\hat \tau_k$, $\tilde \tau_k$ are positive numbers.

We denote $\xi_{1,k}= \xi_1+\tau_kv_k$ (it is the second endpoint of $l_k$). Then $\xi_{1,k}$ is defined by
\begin{align}
\label{ksi_k_def}
\alpha_k=0, \; r_1\alpha_1=\dots=r_{k-1}\alpha_{k-1}= r_{k+1}\alpha_{k+1}=\dots =r_d\alpha_d, \; \alpha_1+\dots +\alpha_d=s=1.
\end{align}
We set 
\begin{align}
\label{psi_j_def_jk} \psi_j(\alpha_1, \, \dots, \, \alpha_d, \, s) = r_j\alpha_j, \quad 1\le j\le d.
\end{align}
Then
\begin{align}
\label{hxi1xi4}
\begin{array}{c}
\psi_j(\xi_1) \stackrel{\eqref{aj12}}{=} \frac{1}{\sum \limits _{i=1}^d 1/r_i}, \quad \psi_j(\xi_4) \stackrel{\eqref{aj34}}{=} \frac{q}{2\sum \limits _{i=1}^d 1/r_i}, \quad 1\le j\le d, \\ \psi_j(\xi_{1,k}) \stackrel{\eqref{ksi_k_def}}{=} \frac{1}{\sum \limits _{i\ne k}1/r_i}, \quad j\ne k.
\end{array}
\end{align}

The set $\tilde D$ is divided into polyhedrons such that the restriction of $\tilde h$ on each of these polyhedrons is affine. We find their vertices with strictly positive $\alpha_j$ and the edges that come out from these vertices.

Denote by $V$ one of these polyhedrons. 
\begin{enumerate}
\item Let $V = \{(\alpha_1, \, \dots, \, \alpha_d, \, s)\in \tilde D:\; \tilde h(\alpha_1, \, \dots, \, \alpha_d, \, s)=r_j\alpha_j\}$, where $j\in I$. We use part 1 of Lemma \ref{oblast}. In the vertices of $V$ with strictly positive coordinates the following equalities hold: $r_1\alpha_1=\dots = r_d\alpha_d$, where $\alpha_1+\dots +\alpha_d=s=1$ or $\alpha_1+\dots +\alpha_d=s=q/2$. These equalities give the points $\xi_1$ and $\xi_4$. The edges that come out from $\xi_1$ are given by
$$
\{(\alpha_1, \, \dots, \, \alpha_d, \, s):\; r_1\alpha_1=\dots = r_d\alpha_d, \; \alpha_1+\dots +\alpha_d=s\in [1, \, q/2]\}
$$
or by \eqref{l_k_def} with $k\ne j$ (see part 1 of Lemma \ref{oblast}); i.e., these are $\hat l_{1,4}$ and $l_k$, $k\ne j$. 
From \eqref{psi_j_def_jk}, \eqref{hxi1xi4} it follows that
\begin{align}
\label{hxi1xi4xi1k}
\tilde h(\xi_1)< \tilde h(\xi_4), \quad \tilde h(\xi_1)< \tilde h(\xi_{1,k}), \quad k\ne j.
\end{align}
Hence the minimum of the function $\tilde h$ on $V$ can be attained only at $\xi_1$; it is equal to 
\begin{align}
\label{1minvh}
\min_V \tilde h =\tilde h(\xi_1)= \psi_j(\xi_1)\stackrel{\eqref{hxi1xi4}}{=}\frac{\langle \overline{r}\rangle}{d}.
\end{align}

\item Let 
\begin{align}
\label{v_case2}
V = \Bigl\{(\alpha_1, \, \dots, \, \alpha_d, \, s)\in \tilde D:\; \tilde h(\alpha_1, \, \dots, \, \alpha_d, \, s)=r_j\alpha_j-\frac 12\cdot \frac{1/p_j-1/q}{1/2-1/q}(s-1)\Bigr\},
\end{align}
where $j\in J$. By part 2 of Lemma \ref{oblast},
the vertices of $V$ with positive coordinates satisfy the equations
$$
r_1\alpha_1-\frac 12\cdot \frac{1/p_1-1/q}{1/2-1/q}(s-1) = \dots = r_d\alpha_d-\frac 12\cdot \frac{1/p_d-1/q}{1/2-1/q}(s-1),
$$
$$
\alpha_1+\dots+\alpha_d=s, \quad s=1 \; \text{or }s = q/2.
$$
For $s=1$ we get $\xi_1$, and for $s=q/2$, we get $\xi_3$.

The edges that come out from $\xi_1$ are either $l_k$ ($k=1, \, \dots, \, d$, $k\ne j$; see \eqref{l_k_def} and part 2 of Lemma \ref{oblast}), or the segment $\hat l_{1,3}$ that joins $\xi_1$ and $\xi_3$. By \eqref{v_case2}, for $s=1$ we have $\tilde h(\alpha_1, \, \dots, \, \alpha_d, \, 1) = r_j\alpha_j$; hence from \eqref{psi_j_def_jk}, \eqref{hxi1xi4} we get $\tilde h(\xi_1)<\tilde h(\xi_{1,k})$, $k\ne j$. Therefore, if $\tilde h(\xi_1)\le \tilde h(\xi_3)$, then $\xi_1$ is the minimum point of $\tilde h$ on $V$, and $\tilde h (\xi_1) = \frac{\langle \overline{r}\rangle}{d}$.

The edges that come out from $\xi_3$ are either $\hat l_{1,3}$ or the segments $\tilde l_k$ ($k=1, \, \dots, \, d$, $k\ne j$; see \eqref{til_l_k_def}). Let $\xi_{3,k}\ne \xi_3$ be the endpoint of the edge $\tilde l_k$. From \eqref{lk_vk}, \eqref{psi_j_def_jk}, \eqref{hxi1xi4} and \eqref{v_case2} it follows that
\begin{align}
\label{hxi2lhxi2k}
\tilde h(\xi_3) < \tilde h(\xi_{3,k}).
\end{align}
Hence, if $\tilde h(\xi_1)\ge \tilde h(\xi_3)$, then $\xi_3$ is the minimum point of $\tilde h$ on $V$, and $\tilde h(\xi_3) = \frac q2 \Bigl( \frac{\langle \overline{r}\rangle}{d} + \frac 1q -\frac{\langle \overline{r}\rangle}{\langle \overline{p}\circ \overline{r}\rangle}\Bigr)$.

Therefore, 
\begin{align}
\label{min_v_case2} \min _V \tilde h = \min \{\tilde h(\xi_1), \, \tilde h(\xi_3)\} = \min \left \{\frac{\langle \overline{r}\rangle}{d}, \, \frac q2 \Bigl( \frac{\langle \overline{r}\rangle}{d} + \frac 1q -\frac{\langle \overline{r}\rangle}{\langle \overline{p}\circ \overline{r}\rangle}\Bigr)\right\};
\end{align}
if, in addition, $\frac{\langle \overline{r}\rangle}{d} \ne \frac q2 \Bigl( \frac{\langle \overline{r}\rangle}{d} + \frac 1q -\frac{\langle \overline{r}\rangle}{\langle \overline{p}\circ \overline{r}\rangle}\Bigr)$, the minimum point on $V$ is unique.

\item Let 
\begin{align}
\label{v_case3}
V = \{(\alpha_1, \, \dots, \, \alpha_d, \, s)\in \tilde D:\; \tilde h(\alpha_1, \, \dots, \, \alpha_d, \, s)=r_j\alpha_j+1/2-s/p_j\}, 
\end{align}
where $j\in K$. By part 3 of Lemma \ref{oblast}, the vertices  of $V$ with positive coordinates satisfy the equations
$$
r_1\alpha_1-\frac{s}{p_1} = \dots = r_d\alpha_d-\frac{s}{p_d},
$$
$$
\alpha_1+\dots+\alpha_d=s, \quad s=1\;\text{or }s=q/2.
$$
For $s=1$ we get $\xi_2$, and for $s=q/2$, we obtain $\xi_3$.

The edges that come out from $\xi_3$ are either $\tilde l_k$ ($k\ne j$; see \eqref{til_l_k_def} and part 3 of Lemma \ref{oblast}) or $\hat l_{2,3}$. The edges that come out from $\xi_2$ are either $\hat l_{2,3}$ or $\hat l_k$ ($k\ne j$; see \eqref{hat_l_k_def}).

Let $\xi_{2,k}\ne \xi_2$ be the endpoint of the edge $\hat l_k$, and let $\xi_{3,k}\ne \xi_3$ be the endpoint of $\tilde l_k$. From \eqref{lk_vk}, \eqref{psi_j_def_jk}, \eqref{hxi1xi4} and \eqref{v_case3} it follows that $\tilde h(\xi_2)< \tilde h(\xi_{2,k})$, $\tilde h(\xi_3)< \tilde h(\xi_{3,k})$.

Hence
\begin{align}
\label{h_v_min_case3}
\begin{array}{c}
\min _V\tilde h = \min\{\tilde h(\xi_2), \, \tilde h(\xi_3)\} = \\ = \min \left\{ \frac{\langle \overline{r}\rangle}{d} + \frac 12 -\frac{\langle \overline{r}\rangle}{\langle \overline{p}\circ \overline{r}\rangle}, \, \frac q2 \Bigl( \frac{\langle \overline{r}\rangle}{d} + \frac 1q -\frac{\langle \overline{r}\rangle}{\langle \overline{p}\circ \overline{r}\rangle}\Bigr)\right\};
\end{array}
\end{align}
if, in addition, $\frac{\langle \overline{r}\rangle}{d} + \frac 12 -\frac{\langle \overline{r}\rangle}{\langle \overline{p}\circ \overline{r}\rangle}\ne \frac q2 \Bigl( \frac{\langle \overline{r}\rangle}{d} + \frac 1q -\frac{\langle \overline{r}\rangle}{\langle \overline{p}\circ \overline{r}\rangle}\Bigr)$, then minimum of $\tilde h$ on $V$ is attained only at one point.

\item Let
\begin{align}
\label{v_case4} V = \{(\alpha_1, \, \dots, \, \alpha_d, \, s)\in \tilde D:\; \tilde h(\alpha_1, \, \dots, \, \alpha_d, \, s)=(1-\lambda_{i,j})r_i\alpha_i + \lambda_{i,j}r_j\alpha_j\},
\end{align}
where $i\in I$, $j\in J\cup K$. By part 4 of Lemma \ref{oblast}, the vertices of $V$ with positive coordinates satisfy the equations
\begin{align}
\label{vert1}
\frac{\alpha_ir_i-\alpha_jr_j}{1/p_i-1/p_j} = \frac{\alpha_ir_i-\alpha_kr_k}{1/p_i-1/p_k}, \quad k\in J\cup K,
\end{align}
\begin{align}
\label{vert2}
\frac{\alpha_ir_i-\alpha_jr_j}{1/p_i-1/p_j} = \frac{\alpha_kr_k-\alpha_jr_j}{1/p_k-1/p_j}, \quad k\in I,
\end{align}
$$
\alpha_1+\dots +\alpha_d=s, \quad s = 1 \;\text{or }s = q/2,
$$
$$
\alpha_ir_i - \alpha_jr_j=0 \;\text{or }\alpha_ir_i - \alpha_jr_j=\frac 12\cdot \frac{1/p_i-1/p_j}{1/2-1/q}(s-1).
$$

For $\alpha_ir_i - \alpha_jr_j=0$, we have 
$$\alpha_1r_1=\dots=\alpha_dr_d, \quad \alpha_1+\dots +\alpha_d=s, \quad s = 1 \;\text{or }s = q/2.$$
For $s=1$ we get $\xi_1$, and for $s=q/2$, we get $\xi_4$.

Let $\alpha_ir_i - \alpha_jr_j=\frac 12\cdot \frac{1/p_i-1/p_j}{1/2-1/q}(s-1)$. For $s=1$ it is equivalent to $\alpha_ir_i - \alpha_jr_j=0$; hence we again obtain the point $\xi_1$. For $s = q/2$ we have $\alpha_ir_i - \alpha_jr_j=\frac q2 (1/p_i-1/p_j)$; this together with \eqref{vert1}, \eqref{vert2}
yields
$$\alpha_1 - \frac{q}{2p_1}=\dots=\alpha_d-\frac{q}{2p_d},\quad \alpha_1+\dots+\alpha_d=s=q/2.$$ These equations define the point $\xi_3$.

The edges that come out from $\xi_1$ are either $l_k$ ($k\ne i$, $j$), or $\hat l_{1,3}$, or $\hat l_{1,4}$. At the edges $l_k$ and $\hat l_{1,4}$ the equality $r_i\alpha_i=r_j\alpha_j$ holds and the function $\tilde h$ is equal to $\alpha_ir_i$. This together with \eqref{psi_j_def_jk}, \eqref{hxi1xi4} yields that $\tilde h(\xi_1)< \tilde h(\xi_4)$, $\tilde h(\xi_1)< \tilde h(\xi_{1,k})$, $k\ne i, \, j$. Therefore, if $\tilde h(\xi_1)\le \tilde h(\xi_3)$, then $\min _{V} \tilde h = \tilde h(\xi_1) = \frac{\langle \overline{r} \rangle}{d}$. 

The edges that come out from $\xi_3$ are either $\tilde l_k$ ($k\ne i$, $j$), or $\hat l_{1,3}$, or $\hat l_{3,4}$. At the edges $\tilde l_k$ the equality $r_i\alpha_i-\frac{q}{2p_i}=r_j\alpha_j-\frac{q}{2p_j}$ holds; hence, at the edge $\tilde l_k$ the function $\tilde h$ is equal to $(1-\lambda_{i,j})r_i\alpha_i+\lambda_{i,j}r_j\alpha_j = r_i\alpha_i +\frac 12 -\frac{q}{2p_i}$. Consequently, by \eqref{lk_vk}, \eqref{psi_j_def_jk} and \eqref{hxi1xi4}, $\tilde h(\xi_3)<\tilde h(\xi_{3,k})$, $k\ne i$, $j$. Hence, if $\tilde h(\xi_3)\le \tilde h(\xi_1)$, this together with $\tilde h(\xi_1) < \tilde h(\xi_4)$ implies that $\min _V \tilde h = \tilde h(\xi_3) = \frac q2 \Bigl( \frac{\langle \overline{r}\rangle}{d} + \frac 1q -\frac{\langle \overline{r}\rangle}{\langle \overline{p}\circ \overline{r}\rangle}\Bigr)$.

We get
\begin{align}
\label{min_v_case4} \min _V \tilde h = \min \{\tilde h(\xi_1), \, \tilde h(\xi_3)\} = \min \left \{\frac{\langle \overline{r}\rangle}{d}, \, \frac q2 \Bigl( \frac{\langle \overline{r}\rangle}{d} + \frac 1q -\frac{\langle \overline{r}\rangle}{\langle \overline{p}\circ \overline{r}\rangle}\Bigr)\right\};
\end{align}
if, in addition, $\frac{\langle \overline{r}\rangle}{d} \ne \frac q2 \Bigl( \frac{\langle \overline{r}\rangle}{d} + \frac 1q -\frac{\langle \overline{r}\rangle}{\langle \overline{p}\circ \overline{r}\rangle}\Bigr)$, the minimum on $V$ is attained only at one point.

\item Similarly we get that if
$$
V = \Bigl\{(\alpha_1, \, \dots, \, \alpha_d, \, s)\in \tilde D:\; \tilde h(\alpha_1, \, \dots, \, \alpha_d, \, s)=(1-\mu_{i,j})r_i\alpha_i + \mu_{i,j}r_j\alpha_j + \frac 12 -\frac s2\Bigr\},
$$
$i\in I\cup J$, $j\in K$, then the vertices of $V$ with positive coordinates are $\xi_1$, $\xi_2$, $\xi_3$ and
\begin{align}
\label{min_v_case5} \min_V \tilde h = \min \{h(\xi_1), \, h(\xi_2), \, h(\xi_3)\} = \min \{\theta_1, \, \theta_2, \, \theta_3\}
\end{align}
(see notation in the theorem); recall that, by the conditions of Theorem \ref{main1}, there is $j_*\in \{1, \, 2, \, 3\}$ such that $\theta_{j_*} = \min _{j\ne j_*} \theta_j$. Therefore, the minimum in \eqref{min_v_case5} is attained only at one point.
\end{enumerate}

So the set $\tilde D$ is divided into closed polyhedrons $V^{(k)}$, $1\le k\le k_0$; each of them is defined by conditions from cases 1--5 above.
\begin{itemize}
\item If $I=\{1, \, \dots, \, d\}$, we get only case 1, and $$\min _{\tilde D} \tilde h = \min _{1\le k\le k_0}\min _{V^{(k)}} \tilde h \stackrel{\eqref{1minvh}}{=} \frac{\langle \overline{r}\rangle}{d};$$ here we do not use  \eqref{dop_usl} in our arguments (see Remark \ref{rem1}).

\item If $I \ne \{1, \, \dots, \, d\}$ and $I\cup J = \{1, \, \dots, \, d\}$, we get only cases 1, 2 and 4. By \eqref{1minvh}, \eqref{min_v_case2}, \eqref{min_v_case4}, $\min _{\tilde D} \tilde h= \min _{1\le k\le k_0}\min _{V^{(k)}} \tilde h = \min \{\theta_1, \, \theta_3\}$, and the minimum on $\tilde D$ is attained only at one point (since $\theta_1\ne \theta_3$ by the assumptions of Theorem \ref{main1}).

\item If $K=\{1, \, \dots, \, d\}$, we get only case 3. By \eqref{h_v_min_case3}, $\min _{\tilde D} \tilde h= \min _{1\le k\le k_0}\min _{V^{(k)}} \tilde h = \min \{\theta_2, \, \theta_3\}$; the minimum on $\tilde D$ is attained only at one point, since, by the conditions of Theorem \ref{main1}, $\theta_2\ne \theta_3$.

\item There remains the case $K\ne \varnothing$, $I\cup J \ne \varnothing$. By the assumptions of Theorem \ref{main1}, there is $j_*$ such that $\theta_{j_*} < \min _{j\ne j_*} \theta_j$. If $V^{(k)}$ contains $\xi_{j_*}$, then, by \eqref{1minvh}, \eqref{min_v_case2}, \eqref{h_v_min_case3}, \eqref{min_v_case4}, \eqref{min_v_case5}, we have $\min _{V^{(k)}} \tilde h = \theta_{j_*}$, and the minimum on $V^{(k)}$ is attained only at $\xi_{j_*}$. If $V^{(k)}$ does not contain $\xi_{j_*}$, we have $\min _{V^{(k)}} \tilde h > \theta_{j_*}$. Hence $\min _{\tilde D} \tilde h= \theta_{j_*}$ and the minimum point is unique.
\end{itemize}

\vskip 0.3cm

{\bf Proof of Theorem \ref{main1} in the general case.}

Now we consider the general case, when $p_i$ may be equal to $2$ or $q$. We define $\overline{p}^N = (p_1^N, \, \dots, \, p_d^N)$ as follows. If $p_j\notin \{2, \, q\}$, we set $p_j^N=p_j$. If $p_j=q$, we set $p_j^N = q + \frac 1N$, and if $p_j=2$, we set $p_j^N = 2 \pm\frac 1N$ (the sign is the same for all $j$; if $K=\{1, \, \dots, \, d\}$, we take ``$-$''; otherwise, we take ``$+$''). For large $N$ we have $p_j^N\notin \{2, \, q\}$, $1\le j\le d$. The function $\tilde h^N$ is defined by the same formula as $\tilde h$, replacing $p_j$ by $p_j^N$. Then $\tilde h^N$ converges to $\tilde h$ uniformly on $\tilde D$. Condition \eqref{dop_usl} for $\overline{p}^N$ holds for large $N$ if it holds for $\overline{p}$.

Let $\xi_t$ ($1\le t\le 4$) be defined by \eqref{xi_k_def14}, \eqref{aj12}, \eqref{aj34}. The set $T\subset \{1, \, 2, \, 3\}$ is defined as follows: if $I = \{1, \, \dots, \, d\}$, then $T=\{1\}$; if $I \ne \{1, \, \dots, \, d\}$ and $I\cup J = \{1, \, \dots, \, d\}\ne K$, then $T=\{1, \, 3\}$; if $K = \{1, \, \dots, \, d\}$, then $T=\{2, \, 3\}$; in the other cases $T = \{1, \, 2, \, 3\}$. Notice that all points $\xi_t$, $t\in T$, are different.

We show that $\min _{\tilde D} \tilde h = \min _{t\in T} \tilde h(\xi_t)$. Indeed, we define $\xi_t^N$ and $T_N$ similarly as $\xi_t$ and $T$, replacing $\overline{p}$ by $\overline{p}^N$. Then $\xi_t^N\underset{N\to \infty}{\to} \xi_t$ and, for large $N$, we have $T_N=T$ (in what follows, we consider only such $N$). We have already proved that $\min_{\tilde D} \tilde h^N = \min _{t\in T}\tilde h^N(\xi_t^N)$. There are $t_*\in T$ and a subsequence $\{N_m\}_{m\in \N}$ such that $\min _{t\in T}\tilde h^{N_m}(\xi_t^{N_m}) = \tilde h^{N_m}(\xi_{t_*}^{N_m})$. Since $\tilde h^N$ converges to $\tilde h$ uniformly on $\tilde D$ and $\xi_t^N\underset{N\to \infty}{\to} \xi_t$, we have $\min _{\tilde D} \tilde h = \tilde h(\xi_{t_*})$. The explicit formula for $\tilde h(\xi_{t_*})$ also follows from formulas for $\tilde h^{N_m}(\xi_{t_*}^{N_m})$; see \eqref{1minvh}, \eqref{min_v_case2}, \eqref{h_v_min_case3}, \eqref{min_v_case4}, \eqref{min_v_case5}.

Now we prove that $\xi_{t_*}$ is the unique minimum point of $\tilde h$. To this end, it siffices to check that there is a number $c = c(\overline{p}, \, q, \, \overline{r}, \, d)>0$ such that for large $m\in \N$
\begin{align}
\label{til_h_n_est_low} \tilde h^{N_m}(\xi) -\tilde h^{N_m}(\xi^{N_m}_{t_*}) \ge c|\xi - \xi^{N_m}_{t_*}|, \quad \xi \in \tilde D
\end{align}
(here $|\cdot|$ is the Euclidean norm on $\R^{d+1}$); this implies that $\tilde h(\xi)-\tilde h(\xi_{t_*})\ge c|\xi-\xi_{t_*}|$, $\xi\in \tilde D$.

Now we prove \eqref{til_h_n_est_low}. We again consider a polyhedron $V=V(m)$, which contains $\xi_{t_*}^{N_m}$, such that $\tilde h^{N_m}|_{V(m)}$ is affine (see cases 1--5 above). It suffices to prove that \eqref{til_h_n_est_low} holds for the points $\xi$ from the edges that come out from $\xi_{t_*}^{N_m}$. Indeed, by the assumptions of Theorem \ref{main1}, $\tilde h(\xi_{t_*}) < \tilde h(\xi_t)$, $t\in T\backslash \{t_*\}$. Since $\tilde h^N$ uniformly converges to $\tilde h$ on $\tilde D$ and $\xi_t^N$ converges to $\xi_t$, this yields that for large $m$ we have $\tilde h^{N_m}(\xi_t^{N_m}) - \tilde h^{N_m}(\xi_{t_*}^{N_m}) \underset{\overline{p}, \, q, \, \overline{r}, \, d}{\gtrsim} |\xi_t^{N_m} - \xi_{t_*}^{N_m}|$. Hence \eqref{til_h_n_est_low} holds at the edges joining $\xi_{t_*}^{N_m}$ and $\xi^{N_m}_t$, $t\in T\backslash \{t_*\}$. Let the edge join $\xi_{t_*}^{N_m}$ and $\xi_4^{N_m}$ (then $\xi_1^{N_m} \in V$; see cases 1 and 4), and $\tilde h^{N_m}(\xi_4^{N_m}) = \frac q2 \tilde h^{N_m}(\xi_1^{N_m}) = \frac q2 \cdot \frac{\langle \overline{r} \rangle}{d}$. Therefore, \eqref{til_h_n_est_low} also holds at this edge. Finally, the edge that comes out from $\xi^{N_m}_{t_*}$ can coincide with $l_k^m$, $\tilde l_k^m$ or $\hat l_k^m$ (these segments are given by formulas similar to \eqref{l_k_def}, \eqref{til_l_k_def} and \eqref{hat_l_k_def}, where $\overline{p}$ is replaced by $\overline{p}^{N_m}$). It was proved that at $l_k^m$, $\tilde l_k^m$ or $\hat l_k^m$ the function $\tilde h^{N_m}$ has the form $\alpha_jr_j + {\rm const}$; the number $s$ at these edges is equal to $1$ or $q/2$. Taking into account \eqref{lk_vk}, \eqref{psi_j_def_jk}, \eqref{hxi1xi4}, we get that \eqref{til_h_n_est_low} holds on $l_k^m$, $\tilde l_k^m$ on $\hat l_k^m$ that come out from $\xi^{N_m}_{t_*}$. This completes the proof.

\section{Proof of Theorems \ref{not_emb} and \ref{main2}}

\renewcommand{\proofname}{\bf Proof of Theorem \ref{not_emb}}

First we prove that if $\frac{\langle \overline{r}\rangle}{d} + \frac 1q - \frac{\langle \overline{r}\rangle}{\langle \overline{r}\circ \overline{p}\rangle}\le 0$, then $W^{\overline{r}}_{\overline{p}}(\mathbb{T}^d)$ is not compactly embedded into $L_q(\mathbb{T}^d)$. It can be checked by induction on $d$. For $d=1$, this is a well-known result.

Suppose that for $d-1$ the assertion is proved.

Let $\frac{\langle \overline{r}\rangle}{d} + \frac 1q - \frac{\langle \overline{r}\rangle}{\langle \overline{r}\circ \overline{p}\rangle}\le 0$. Then there is $j\in \{1, \, \dots, \, d\}$ such that $p_j<q$.

First we suppose that \eqref{dop_usl} holds. Taking into account that $p_j<q$ for some $j$, we get that $\min _D h\le \frac{\langle \overline{r}\rangle}{d} + \frac 1q - \frac{\langle \overline{r}\rangle}{\langle \overline{r}\circ \overline{p}\rangle}\le 0$ for $q\le 2$, $$\min _{\tilde D} \tilde h \le \frac q2 \left( \frac{\langle \overline{r}\rangle}{d} + \frac 1q - \frac{\langle \overline{r}\rangle}{\langle \overline{r}\circ \overline{p}\rangle}\right)\le 0$$ for $q>2$ (see Theorems \ref{main1}, \ref{main}). Hence $d_n(W^{\overline{r}}_{\overline{p}}(\mathbb{T}^d), \, L_q(\mathbb{T}^d)) \underset{\overline{r}, d, \overline{p}, q}{\gtrsim} 1$ (see Remark \ref{h_otr}) and the embedding is not compact.

Suppose that \eqref{dop_usl} fails; i.e., there is $j\in \{1, \, \dots, \, d\}$ such that 
\begin{align}
\label{ne_cond}
\sum \limits _{i=1}^d \frac{1}{r_i} \left(\frac{1}{p_i} - \frac{1}{p_j}\right) \ge 1.
\end{align}
We denote $\overline{r}_j =(r_1, \, \dots, \, r_{j-1}, \, r_{j+1}, \, \dots, \, r_d)$, $\overline{p}_j =(p_1, \, \dots, \, p_{j-1}, \, p_{j+1}, \, \dots, \, p_d)$. Then \eqref{ne_cond} is equivalent to 
\begin{align}
\label{ovr_r_j_d1}
\frac{\langle \overline{r}_j\rangle}{d-1} + \frac{1}{p_j} - \frac{\langle \overline{r}_j\rangle}{\langle \overline{r}_j\circ \overline{p}_j\rangle}\le 0.
\end{align}

If $p_j\le q$, then 
\begin{align}
\label{ovr_r_j_d1q}
\frac{\langle \overline{r}_j\rangle}{d-1} + \frac{1}{q} - \frac{\langle \overline{r}_j\rangle}{\langle \overline{r}_j\circ \overline{p}_j\rangle}\le 0
\end{align}
and $W^{\overline{r}_j}_{\overline{p}_j}(\mathbb{T}^{d-1})$ in not embedded compactly into $L_q(\mathbb{T}^{d-1})$ by induction assumption. This implies that $W^{\overline{r}}_{\overline{p}}(\mathbb{T}^{d})$ in not embedded compactly into $L_q(\mathbb{T}^{d})$.

Let $p_j>q$. We set $\overline{p}^*=(p_1, \, \dots, \, p_{j-1}, \, q, \, p_{j+1}, \, \dots, \, p_d)$. Then $\frac{\langle \overline{r}\rangle}{d} + \frac 1q - \frac{\langle \overline{r}\rangle}{\langle \overline{r}\circ \overline{p}^*\rangle}<0$ and, hence, $W^{\overline{r}}_{\overline{p}^*}(\mathbb{T}^{d})$ is not bounded in $L_q(\mathbb{T}^{d})$ \cite[Theorem 1]{galeev_emb}. On the other hand, $$W^{\overline{r}}_{\overline{p}^*}(\mathbb{T}^{d}) \subset\{f\in \mathaccent'27{\cal S}'(\mathbb{T}^d):\; \|\partial _j^{r_j}f\|_{L_{q}(\mathbb{T}^d)}\le 1\};$$
the right-hand side is bounded in $L_q(\mathbb{T}^{d})$. We get a contradiction. Hence the case $p_j>q$ is impossible.
\begin{proof}
Let $j\in \{1, \, \dots, \, d\}$ be such that \eqref{ne_cond} holds; it is equivalent to \eqref{ovr_r_j_d1}. By assumptions of Theorem \ref{not_emb}, $p_j\le q$. Hence \eqref{ovr_r_j_d1q} holds. It was proved that $W^{\overline{r}_j}_{\overline{p}_j}(\mathbb{T}^{d-1})$ is not embedded compactly into $L_q(\mathbb{T}^{d-1})$; therefore, $W^{\overline{r}}_{\overline{p}}(\mathbb{T}^{d})$ is not embedded compactly into $L_q(\mathbb{T}^{d})$.
\end{proof}

\renewcommand{\proofname}{\bf Proof of Theorem \ref{main2}}

\begin{proof}
We apply Theorem \ref{main} and write the functions $h$ for $q\le 2$ and $\tilde h$ for $q>2$.

Let $q\le 2$. Then, by \eqref{r2lp2p1} and \eqref{h_def_00},
$$
h(\alpha_1, \, \alpha_2) = \begin{cases} r_1\alpha_1 & \text{for } r_1\alpha_1-r_2\alpha_2\ge 0, \\ (1-\lambda)r_1\alpha_1+\lambda r_2\alpha_2 & \text{for } r_1\alpha_1-r_2\alpha_2\le 0\end{cases}
$$
(the case $h(\alpha_1, \, \alpha_2)= r_2\alpha_2 +1/q-1/p_2 > (1-\lambda)r_1\alpha_1+\lambda r_2\alpha_2$ is possible only for $r_2\alpha_2> r_1\alpha_1+1/p_2-1/p_1$, which contradicts with \eqref{r2lp2p1}). Since $\frac{\langle \overline{r}\rangle}{2}\ne \lambda r_2$, the minimum of $h$ is attained at one of the following points: $\Bigl(\frac{1/r_1}{1/r_1+1/r_2}, \, \frac{1/r_2}{1/r_1+1/r_2}\Bigr)$ or $(0, \, 1)$. This implies \eqref{dn_wrp_r2}.

Let $q>2$. Notice that $\hat s\in [1, \, q/2]$.

For $p_2\ge 2$, by Lemma \ref{oblast}
$$
\tilde h(\alpha_1, \, \alpha_2, \, s) = \begin{cases} r_1\alpha_1 & \text{for } r_1\alpha_1-r_2\alpha_2\ge 0, \\ (1-\lambda)r_1\alpha_1+\lambda r_2\alpha_2 & \text{for } \frac12\cdot \frac{1/p_1-1/p_2}{1/2-1/q}(s-1)\le r_1\alpha_1-r_2\alpha_2\le 0, \\
r_2\alpha_2-\frac 12\cdot \frac{1/p_2-1/q}{1/2-1/q}(s-1) &  \text{for } r_1\alpha_1-r_2\alpha_2\le \frac12\cdot \frac{1/p_1-1/p_2}{1/2-1/q}(s-1).\end{cases}
$$
Since $\theta_1\ne \theta_2$, the minimum of this function is attained at one of the following points: $\Bigl(\frac{1/r_1}{1/r_1+1/r_2}, \, \frac{1/r_2}{1/r_1+1/r_2}, \, 1\Bigr)$ or $(0, \, \hat s, \, \hat s)$.

For $p_2< 2$, by Lemma \ref{oblast}
$$
\tilde h(\alpha_1, \, \alpha_2, \, s) = \begin{cases} r_1\alpha_1 & \text{for } r_1\alpha_1-r_2\alpha_2\ge 0, \\ (1-\lambda)r_1\alpha_1+\lambda r_2\alpha_2 & \text{for } \frac12\cdot \frac{1/p_1-1/p_2}{1/2-1/q}(s-1)\le r_1\alpha_1-r_2\alpha_2\le 0, \\
(1-\mu)r_1\alpha_1+\mu r_2\alpha_2-\frac 12 (s-1) & \text{for } r_1\alpha_1-r_2\alpha_2\le \frac12\cdot \frac{1/p_1-1/p_2}{1/2-1/q}(s-1)\end{cases}
$$
(the case $\tilde h(\alpha_1, \, \alpha_2, \, s) = r_2\alpha_2 + \frac 12 - \frac{s}{p_2}>(1-\mu)r_1\alpha_1+\mu r_2\alpha_2-\frac 12 (s-1)$ is impossible by \eqref{r2lp2p1}).

By \eqref{theta_j_st_l_min}, the minimum of the function $\tilde h$ is attained at one of the following points: $(0, \, \hat s, \, \hat s)$, $(0,\, 1, \, 1)$ and $\Bigl(\frac{1/r_1}{1/r_1+1/r_2}, \, \frac{1/r_2}{1/r_1+1/r_2}, \, 1\Bigr)$. 

This implies the estimates in part 2 of Theorem \ref{main2}.
\end{proof}

\end{document}